\documentclass[reqno,11pt]{amsart}
\usepackage{euscript}
\usepackage{amssymb}
\usepackage{amscd}
\usepackage{amsmath}
\usepackage{epic}
\usepackage{graphics}
\usepackage{epsfig}
\usepackage{color}
\usepackage{setspace}
\usepackage[backref]{hyperref}

\numberwithin{equation}{section}
\newtheoremstyle{my}{1.5em}{0.5em}{\em}{}{\sc}{.}{0.5em}{}
\newtheoremstyle{your}{1.5em}{0.5em}{}{}{\sc}{.}{0.5em}{}

\theoremstyle{my}

\theoremstyle{my}
\newtheorem{thm}{Theorem}[section]
\newtheorem{Theorem}[thm]{Theorem}
\newtheorem*{Theorem*}{Theorem}
\newtheorem{Corollary}[thm]{Corollary}

\newtheorem{cor}[thm]{Corollary}
\newtheorem*{corollary*}{Corollary}
\newtheorem{Lemma}[thm]{Lemma}

\newtheorem*{conjecture*}{Conjecture}

\newtheorem*{question*}{Question}

\newtheorem*{definitions*}{Definitions}

\newtheorem*{rem*}{Remark}

\newtheorem*{remark*}{Remark}

\newtheorem*{remarks*}{Remarks}
\newtheorem*{example*}{Example}
\newtheorem*{examples*}{Examples}

\newtheorem*{convention*}{Convention}
\newtheorem*{conventions*}{Conventions}
\newtheorem*{Note*}{Note}
\newtheorem*{exercise*}{Exercise}
\newtheorem*{bibliographical-note*}{Bibliographical note}

\theoremstyle{your}
\newtheorem{Remark}[thm]{Remark}

\newcommand{\Acknowledgements}{{\em Acknowledgements.} }

\newcommand{\R}{\mathbb{R}}
\newcommand{\Z}{\mathbb{Z}}
\newcommand{\Q}{\mathbb{Q}}
\newcommand{\C}{\mathbb{C}}

\newcommand{\pa}{\partial}

\newcommand{\ind}{\mathrm{ind}}

\newcommand{\rk}{\mathrm{rank}}
\renewcommand{\ker}{\mathrm{ker}}

\newcommand{\st}{\mathrm{st}}
\newcommand{\cu}{\mathrm{cu}}

\newcommand{\om}{\omega} 
\newcommand{\SI}{\mathrm{SI}}

\newcommand{\const}{\mathrm{const}}
\newcommand{\Int}{\operatorname{int}}

\newcommand{\p}{\partial }
\newcommand{\eps}{\epsilon}
\newcommand{\wt}{\widetilde}
\newcommand{\wh}{\widehat}

\title{Constructing exact Lagrangian immersions with few double points}

\author{Tobias Ekholm}
\thanks{T.E. is partially supported by Swedish Research Council Grant 2012-2365 and by the Knut and Alice Wallenberg Foundation as a Wallenberg Scholar.} 
\address{Tobias Ekholm, Department of Mathematics, Uppsala University, Box 480, 751 06 Uppsala, Sweden, and Institute Mittag-Leffler, Aurav 17, 182 60 Djursholm, Sweden}
\author{Yakov Eliashberg}
\thanks{Y.E. is partially supported by   NSF grant DMS-1205349}
\address{Yakov Eliashberg, Department of Mathematics, Stanford University, Stanford, CA 94305-2125 U.S.A.}
\author{Emmy Murphy}
\thanks{E.M. is partially supported by NSF grant DMS-0943787}
\address{Emmy Murphy, Department of Mathematics, Massachusetts Institute of Technology, 77 Massachusetts Ave.,  Cambridge, MA 02139  U.S.A.}
\author{Ivan Smith}
\thanks{I.S. is partially supported by grant ERC-2007-StG-205349 from the European Research Council.}
\address{Ivan Smith, Centre for Mathematical Sciences, University of Cambridge,  Wilberforce Road, Cambridge CB3 0WB, England.}

\date{March 2013}

\begin{document}
\thispagestyle{empty}
\setcounter{tocdepth}{1}

\begin{abstract}
We establish, as an application of the results from \cite{EliMur},   an $h$-principle for exact Lagrangian immersions with transverse self-intersections and the minimal, or  near-minimal number of double points. One corollary of our result is that any orientable closed $3$-manifold admits an exact Lagrangian immersion into standard symplectic $6$-space $\R^6_\st$ with exactly one transverse double point. Our construction  also yields a Lagrangian embedding  $S^1\times S^2\to\R^6_\st$ with vanishing Maslov class.
\end{abstract}

\maketitle

\section{Introduction}\label{Sec:intr}
\subsection*{Lagrangian self-intersections}
In this paper we study the problem of constructing Lagrangian immersions with the minimal possible number of transverse self-intersection points. It is well known that the existence of a Lagrangian embedding imposes strong topological constraints (e.g.~Gromov's theorem about non-existence of exact Lagrangian submanifolds in standard symplectic $2n$-space, $\R^{2n}_\st=(\R^{2n},\sum_{i=1}^{n} dx_i\wedge dy_i)$), and also that, in many cases, two Lagrangian submanifolds must intersect in more points than is suggested by topological intersection theory alone (e.g.~results confirming Arnold's conjectures). In view of this, it was expected that there should be similar lower bounds on the minimal number of double points for Lagrangian immersions, e.g.~that an exact Lagrangian immersion of an $n$-torus into $\R^{2n}_{\st}$ with transverse self-intersections would have at least $2^{n-1}$ double points.  

Bounds of this kind have been proved for Lagrangian immersions satisfying additional conditions.
For instance, it was shown in \cite{EESa,EESori} that any self-transverse exact Lagrangian immersion $f\colon L\to\R^{2n}_\st$ of a closed $n$-manifold, for which the Legendrian lift $\tilde f\colon L\to \R^{2n}_\st\times\R$  has Legendrian homology algebra that admits an augmentation, satisfies the following analogue of the Morse inequalities: the number of double points of $f$ is at least $\frac{1}{2}\sum_{j=0}^{n}\rk(H_{j}(L))$.

Here we prove a surprising result in sharp contrast to such lower bounds: if no extra constraints are imposed then Lagrangian immersions into symplectic manifolds of dimension $>4$ with \emph{nearly the minimal} number of self-intersection points satisfy a certain $h$-principle. Let us  introduce the  notation
needed to state the result. Let $X$ be an oriented  $2n$-manifold and $f\colon L\to X$   a  proper smooth immersion of a   connected $n$-manifold. If $L$ is non-compact we will assume that $f$ is an embedding outside of a compact set. Following Whitney \cite{Whitney}, if $f$ is self-transverse we assign a self-intersection number $I({f})$ to $f$, where $I({f})$ is the mod 2 number of self-intersection points if $n$ is odd  or $L$  is non-orientable, and the algebraic number of self-intersection points counted with their intersection signs if $n$ is even and $L$ is orientable. Note that $I(f)$ depends only on the regular homotopy class of $f$ provided that $L$ is closed or that the regular homotopy is an isotopy at infinity. When $n>2$  and $X$ is simply connected, a theorem of Whitney \cite{Whitney} asserts that  an immersion $f$ is regularly homotopic to an embedding if and only if $I(f)=0$.
If $f\colon L\to X$ is a self-transverse immersion, then we let $\SI(f)\ge 0$ denote its total number of double points. Clearly, $\SI(f)\geq |I(f)|$. 

Similarly, given an orientable $(2n-1)$-dimensional manifold $Y$ and an $(n-1)$-dimensional manifold $\Lambda$, consider a smooth regular homotopy $h_t\colon\Lambda\to Y$, $0\le t\le 1$, which  connects embeddings $h_0$ and $h_1$ and  is an isotopy outside of a compact set. Let $H\colon \Lambda\times[0,1]\to Y\times[0,1]$ be given by  $H(x,t)=(h_t(x), t)$; then $H$ is an immersion. We say that the regular homotopy $h_t$ has transverse intersections if the immersion $H$ has transverse double points. In this case we define $I(\{h_t\}_{t\in[0,1]}):=I(H)$ and $\SI(\{h_t\}_{t\in[0,1]}):=\SI(H)$. Note that if $n$ is even and $\Lambda$ is orientable then $I(\{h_t\}_{t\in[0,1]})=-I(\{h_{1-t}\}_{t\in[0,1]})$.

 Consider next a Lagrangian regular homotopy, $f_t\colon L\to X$, $0\le t\le 1$, and write $F\colon L\times[0,1]\to X$ for $F(x,t)=f_t(x)$. Let $\alpha$ denote the 1-form on $L\times[0,1]$ defined by the equation $\alpha:=\iota_{\pa/\pa t}(F^*\omega)$, where $\iota$ denotes contraction and $t$ is the coordinate on the second factor of $L\times[0,1]$. Then the restrictions $\alpha_t:=\alpha|_{L\times \{t\}}$ are closed for all $t$. We call the Lagrangian regular homotopy $f_t$ a \emph{Hamiltonian regular homotopy} if the cohomology class $[\alpha_t]\in H^1(L)$ is independent of $t$. The following theorem is our main result:

\begin{thm}\label{thm:min-intersect}
Suppose that $X$ is a simply connected $2n$-dimensional symplectic manifold, $n>2$. If $n=3$ we assume further that $X$ has infinite Gromov width (that is, it admits a symplectic embedding of the standard ball $B^6(R)$ for any large $R$). If $f\colon L\to X$ is a Lagrangian immersion then there exists a Hamiltonian regular homotopy $f_t\colon L\to X$, $0\le t\le 1$, with $f_0=f$ such that $f_1$ is self-transverse and 
\[
\SI(f_1) =
\begin{cases}  
1,&  \text{if $n$ is odd or $L$ is non-orientable and $I(f_0)=1$};\\
2,&  \text{if $n$ is odd or $L$ is non-orientable and $I(f_0)=0$};\\
|I(f_0)|,& \text{if $n$ is even,  $L$ is orientable and $(-1)^{\frac n2}I({f_0})<0$};\\
|I(f_0)|+2,& \text{if $n$ is even,  $L$ is orientable and $(-1)^{\frac n2}I({f_0})\ge 0$.}
\end{cases}
\]
\end{thm}
\begin{Remark}
There is a version of Theorem \ref{thm:min-intersect} in the non-simply connected case where  the intersection number $I(f_0)$ is defined as an element of the group ring of $\pi_1(X)$ and where $|I(f_0)|$ denotes an appropriate norm on this ring. For simplicity, we focus on the simply connected case in this paper.
\end{Remark}  
Theorem \ref{thm:min-intersect} is proved in Section \ref{ssec:proofmin-intersect} as an application of results in \cite{EliMur}. 

This result has the following consequences for  exact Lagrangian immersions into $\R^{2n}_\st$.  
Let $L$ be an $n$-dimensional closed manifold. Recall that according to Gromov's $h$-principle for Lagrangian immersions the triviality of the complexified tangent bundle $TL\otimes\C$ is a necessary and sufficient condition for the existence of an exact Lagrangian immersion $L\to\R^{2n}_\st$, while exact Lagrangian regular homotopy classes are in natural one to one correspondence with homotopy classes of trivializations of $TL\otimes\C$. We write $s(L)$ for the minimal number of double points of a self-transverse exact Lagrangian immersion $L\to\R^{2n}_{\st}$. Given a homotopy class $\sigma$ of trivializations of $TL\otimes\C$, the refined invariant $s(L,\sigma)$ denotes the minimal number of double points of an exact self-transverse Lagrangian immersion $L\to\R^{2n}_{\st}$ representing the exact Lagrangian regular homotopy class $\sigma$.
 
\begin{cor}\label{cor:immersion-Rn}
Let $L$ be an $n$-dimensional closed manifold with $TL\otimes\C$ trivial and let $\sigma$ be a homotopy class of trivializations of $TL\otimes\C$. Then the following hold:
\begin{enumerate}
\item  If $n>1$ is odd or if $L$ is non-orientable, then $s(L,\sigma) \in \{1,2\}$.
\item  If $n=1$, then $s(L)=1$ and there exist $\sigma$ with $s(L,\sigma)=d$ for any integer $d>0$; if $n=3$, then $s(L)=1$ and for one of the two regular homotopy classes $\sigma$, $s(L,\sigma)=2$. 
\item  If $n$ is even and $L$ is orientable, then for $\chi(L)<0$, $s(L,\sigma) = \frac{1}{2}|\chi(L)|$, and for
$\chi(L)\geq  0$, either $s(L,\sigma) = \frac{1}{2}\chi(L)$ or $s(L,\sigma) = \frac{1}{2}\chi(L)+2$.
\end{enumerate}
\end{cor}

Corollary \ref{cor:immersion-Rn} is proved in Section \ref{ssec:proofimmersion-Rn}; in Section  \ref{ssec:more-Rn} we  then give more detailed information about the case when $n$ is odd, respectively discuss the Lagrangian embeddings obtained from these immersions by  Lagrange surgery. The case $n=2$ in both (1) and (3) above does not follow from our proof of Theorem \ref{thm:min-intersect}; rather, these are results of Sauvaget \cite{Sauvaget}, who gave direct geometric constructions of self-transverse exact Lagrangian immersions of both oriented and non-oriented surfaces. In particular, Sauvaget constructed, as the key point for his result, an exact immersed genus two surface in $\C^2$ with exactly one double point. In Appendix \ref{Sec:explicit} we describe a higher dimensional analogue of that construction. 

It is interesting to  compare Corollary \ref{cor:immersion-Rn} with the results of \cite{ES1,ES2} which show that any homotopy $n$-sphere $\Sigma$ that admits a Lagrangian immersion into $\R^{2n}_{\st}$ with exactly one transverse double point of even Maslov grading and with induced trivialization of the complexified tangent bundle homotopic to that of the Whitney immersion of the standard $n$-sphere\footnote{i.e. the trivializations $\sigma\colon T S^n \otimes\C\to\C^n$ and $\wt\sigma\colon T \Sigma \otimes\C\to\C^n$ are related as
$\wt\sigma=\sigma\circ( h\otimes \C)$, where $h\colon T\Sigma\to TS$ is a bundle isomorphism covering a homotopy equivalence $\Sigma\to S^n$.}, must bound a parallelizable $(n+1)$-manifold. If $n$ is even, both the Maslov grading condition and the homotopy condition are automatically satisfied, and moreover the standard $n$-sphere is the only homotopy $n$-sphere that bounds a parallelizable $(n+1)$-manifold. Thus, if $n$ is even the standard $n$-sphere is the only homotopy $n$-sphere that admits a self-transverse Lagrangian immersion into Euclidean space with only one double point. This means in particular that in the case when $\dim(L)$ is even and $\chi(L)>0$, $s(L)$ is generally not determined by the homotopy type of $L$.  The following result constrains the homotopy type of a manifold for which this phenomenon may occur.

\begin{thm}\label{thm:hmtpytype} Let $L$ be an even dimensional spin manifold with $\chi(L)>0$. If $s(L)=\frac{1}{2}\chi(L)$ then $\pi_1(L) = 1$ and $H_{2k+1}(L) = 0$ for all $k$. In particular if $\dim L > 4$ then $L$ has the homotopy type of a CW-complex with $\chi(L)$ even-dimensional cells and no odd-dimensional cells.
\end{thm} 

Theorem \ref{thm:hmtpytype} is proved in Section \ref{ssec:proofhmtpytype}. The proof uses lifted linearized Legendrian homology, introduced in \cite{ES2} following ideas in \cite{Damian}. Note that this result can be viewed as an obstruction to an $h$-principle for exact Lagrangian immersions having the minimal, rather than near-minimal, number of self-intersection points. 

\subsection*{Surgery, Lagrangian embeddings and the Maslov class}
In \cite{Polterovich} Polterovich describes a local  Lagrangian  surgery construction which resolves   a double point of a Lagrangian immersion.
  Let $Q_+$ denote the manifold $S^1 \times S^{n-1}$, and $Q_-$ the mapping torus of an orientation-reversing involution of $S^{n-1}$.  Given  a Lagrangian immersion $f\colon L \rightarrow X$ of an oriented $n$-manifold with a single double point $p$, \cite[Propositions 1 \& 2]{Polterovich} imply:
\begin{enumerate}
\item if $n$ is odd,  there are Lagrangian embeddings $L \# Q_{\pm} \rightarrow X$; 
\item if $n$ is even, there is a Lagrangian embedding $L \# Q_{\epsilon} \rightarrow X$, where the sign of $\epsilon$ is given by $(-1)^{n(n-1)/2 + 1} \textrm{sign}(p)$,
\end{enumerate}
where sign$(p) \in \{\pm 1\}$ denotes the intersection index of the double point.   Combining this with Corollary \ref{cor:immersion-Rn} yields:

\begin{Corollary} \label{Cor:embedded}
Let $Y$ be a closed orientable 3-manifold. Then there is a Lagrangian embedding $Y \# (S^1 \times S^2) \rightarrow \R^6_\st$.
\end{Corollary}

The question of determining the minimal number $k$ for which there is a Lagrangian embedding $Y \# k (S^1\times S^2) \rightarrow \R^6_\st$ was raised explicitly by Polterovich \cite[Remark 2]{Polterovich}.  The construction in Appendix \ref{Sec:explicit}, in combination with Lagrange surgery, gives an \emph{explicit} Lagrangian embedding of $Q_+ \# Q_+ \# Q_+$ into $\R^{2n}_{\st}$ for any odd $n\geq 3$.  

In another direction, in each odd dimension $n=2k+1$ our construction
yields  a Lagrangian immersion of the sphere $S^{n} \rightarrow \R^{2n}_{\st}$ with a single double point of Maslov grading $1$ (note the double point of the Whitney sphere has Maslov grading $n$). Using  Lagrange surgery we conclude:
\begin{Corollary}\label{cor:Maslov}
There exists a Lagrangian embedding $S^1 \times S^{2k} \rightarrow \R^{4k+2}_{\st}$ for which the generator of the first homology of positive action has non-positive  Maslov index $2-2k$. In particular, 
there exists a Lagrangian embedding $S^1 \times S^{2} \rightarrow \R^{6}_{\st}$ with zero Maslov class.
\end{Corollary}  

The (non-)existence question for  Lagrangian embeddings into $\R^{2n}_\st$ with vanishing Maslov class is a well-known problem in symplectic topology.  Viterbo proved in \cite{Viterbo} that if $L$ admits a metric of non-positive sectional curvature then any Lagrangian embedding $L \rightarrow \R^{2n}_{\st}$ has non-trivial Maslov class, whilst Fukaya, Oh, Ohta and Ono infer the same conclusion  in Theorem $K$ of  \cite{FOOO} whenever  $L$ is a spin manifold with $H^2(L;\Q)=0$. Corollary \ref{cor:Maslov} shows that the  hypotheses in these theorems play more than a technical role; in particular, the assumption  on second cohomology in the latter result cannot be removed. 
Note that by taking products of the Maslov zero $S^1\times S^2$ in $\R^6_{\st}$, one obtains Maslov zero Lagrangian embeddings in $\R^{6k}_{\st}$ for every $k> 1$. 
We point out that the Lagrangian embedding in Corollary \ref{cor:Maslov} is not monotone. In the monotone case, the Maslov index of the generator of the first homology of positive action necessarily equals $2$, see 
\cite[Proposition 2.10]{Fukaya} (and references therein) and \cite[Theorem 1.5 (b)]{Damian}.

 \subsection*{Plan of the paper}
The paper is organized as follows. In Section \ref{sec:loose-knots} we recall the notions and the main results of the theory of loose Legendrian knots from \cite{Mur}. In Section \ref{sec:conical} we formulate the $h$-principle for Lagrangian embeddings with conical singularities, established in \cite{EliMur}, and deduce from it its own generalization: in the presence of a conical singularity, Lagrangian immersions with the minimal number of self-intersections abide an $h$-principle. Theorem \ref{thm:min-intersect} is then proved in Section \ref{ssec:proofmin-intersect} as an application of this $h$-principle. Corollary \ref{cor:immersion-Rn} and related explicit results about Lagrangian immersions into $\R^{2n}_{\st}$ are proved in Sections \ref{ssec:proofimmersion-Rn}, \ref{ssec:more-Rn}, and \ref{ssec:proofhmtpytype}.   
As is typical with $h$-principles, Theorem \ref{thm:min-intersect} does not provide explicit constructions of Lagrangian immersions with the minimal number of double points. In Appendix \ref{Sec:explicit} we complement Theorem \ref{thm:min-intersect} by constructing an explicit exact Lagrangian immersion of $P=(S^{1}\times S^{n-1})\#(S^{1}\times S^{n-1})$ into $\R^{2n}_{\st}$ with exactly one transverse double point (in particular, yielding an immersion in each dimension violating the Arnold-type bound which pertains in the presence of linearizable Legendrian homology algebra). Our construction is a generalization of Sauvaget's construction in the case $n=2$.  The construction of the appendix seems to have  no elementary relation to the arguments earlier in the paper;  in general, our immersions with few double points are obtained by first introducing many double points and then canceling them in pairs, whilst in the Appendix considerable effort is expended to keep the Lagrangians swept by appropriate Legendrian isotopies embedded.

\Acknowledgements  T.E. and I.S. are grateful to Fran\c{c}ois Laudenbach for drawing their attention to Sauvaget's work.  Y.E. is grateful to the Simons Center for Geometry and Physics where a part of this paper was completed. The authors thank Alexandr Zamorzaev for pointing out a gap in the original proof of Lemma \ref{lm:near-cone}.

\section{Loose Legendrian  knots}\label{sec:loose-knots}
The theory of loose Legendrian knots and the $h$-principle for Lagrangian caps with loose Legendrian ends, developed in \cite{Mur} and \cite{EliMur}, respectively, are crucial for the proof of our main result, Theorem \ref{thm:min-intersect}. In this section we recall the concepts and results from these theories that will be used in later sections.   

\subsection{Stabilization}\label{ssec:stablization}
We start with a discussion of  the \emph{stabilization construction} for Legendrian submanifolds, see \cite{Eliash-Stein, CieEli-Stein, Mur}.

Consider standard contact $\R^{2n-1}$:
$$
\R^{2n-1}_\st=\left(\R^{2n-1} \ , \ \xi_\st=\ker\left(dz-\sum\limits_1^{n-1}y_i dx_i\right)\right),
$$
where  $(x_1,y_1,\dots,x_{n-1},y_{n-1},z)$ are coordinates in $\R^{2n-1}$, with the Legendrian  coordinate subspace $\Lambda_{0}\subset \R^{2n-1}_{\st}$ given by
\[
\Lambda_0=\left\{(x_1,y_1,\dots,x_{n-1},y_{n-1},z)\colon 
x_1=y_2=y_3=\dots=y_{n-1}=0,\; z=0\right\}.
\]
Then $(\R^{2n-1}_{\st},\Lambda_{0})$ is a local model for any Legendrian submanifold in a contact manifold. More precisely, if $\Lambda\subset Y$ is any Legendrian $(n-1)$-submanifold of a contact $(2n-1)$-manifold $Y$ then any point $p\in\Lambda$ has a neighborhood $\Omega\subset Y$ that admits a map
\[
\Phi\colon (\Omega,\Lambda\cap \Omega)\to (\R^{2n-1}_{\rm st},\Lambda_0),\quad \Phi(p)=0,
\]
which is a contactomorphism onto a neighborhood of the origin.

We will carry out the stabilization construction in a local model that is slightly different from $(\R^{2n-1}_{\st},\Lambda_0)$, which we discuss next. Let
$F\colon\R^{2n-1}_\st\to\R^{2n-1}_\st$ denote the contactomorphism,
\[
F(x_{1},y_1,\,\dots,\,x_{n-1},y_{n-1},\,z)=
\left(x_1+\frac12 y_1^2,y_1,\, x_2,y_2,\,\dots,\,x_{n-1},y_{n-1},\, z+\frac13 y_1^3 \right).
\]
Then $F$ maps $\Lambda_0$ to
\[
\Lambda_\cu =\left\{(x_1,y_1,\dots,x_{n-1},y_{n-1},z)\colon 
x_1=\frac{1}{2}y_1^{2},\; y_2=\dots=y_{n-1}=0,\; z=\frac{1}{3}y_1^{3}\right\}
\]   
In the language of Appendix \ref{ssec:frontslice}, the front $\Gamma_\cu$ of $\Lambda_\cu$ in $\R^{n-1}\times\R$ is the product of $\R^{n-2}\subset\R^{n-1}$ and the standard cusp 
$\{(x_1,z)\colon 9z^2= 8x_1^3\}$ in $\R\times\R$. In particular, the two branches of the front are graphs of the functions $\pm h$, where 
\[
h(x)=h(x_1,\dots,x_{n-1})= \tfrac{2\sqrt{2}}3\,x_1^{\frac{3}{2}},
\]
defined on the half-space $\R^{n-1}_+:=\{x=(x_1,\dots,x_{n-1})\colon x_1\geq 0\}$.

Let $U$ be a domain with smooth boundary contained in the interior of $\R^{n-1}_{+}$, $U\subset\Int(\R^{n-1}_+)$. Pick a non-negative function $\phi\colon \R^{n+1}_+\to\R$ with the following properties: $\phi$ has compact support in $\Int(\R^{n-1}_+)$,
the function $\wt\phi(x):=\phi(x)-2h(x)$ is Morse, $U =\wt\phi^{-1}([0,\infty))$, and $0$ is a regular value of $\wt\phi$. Consider the front $\Gamma_\cu^U$ in $\R^{n-1}\times\R$ obtained from $\Gamma_\cu$ by replacing the lower branch of $\Gamma_\cu$, i.e.~the graph 
$z=-h(x)$, by the graph $z=\phi(x)-h(x)$.  Since $\phi$ has compact support, the front  $\Gamma_\cu^U$ coincides with $\Gamma_\cu$ outside a compact set. Consequently, the Legendrian embedding $\Lambda_\cu^U\colon \R^{n-1}\to\R^{2n-1}$ defined by the front $\Gamma_\cu^U$ coincides with $\Lambda_\cu$ outside a compact set.

\begin{Lemma}[\cite{CieEli-Stein, Mur}]\label{lm:stab}
There exists a compactly supported Legendrian regular homotopy $\Lambda_{\cu;\,t}$, $t\in[0,1]$ connecting $\Lambda_\cu$ to $\Lambda^{U}_\cu$ with $\SI\left(\{\Lambda_{\cu;\,t}\}_{t\in[0,1]}\right)$ equal to the number of critical points of $\phi(x)-h(x)$, and with $I\left(\{\Lambda_{\cu;\,t}\}_{t\in[0,1]}\right)= (-1)^{k-1}\chi(U)$ if $n=2k$  and $I\left(\{\Lambda_{\cu;\,t}\}_{t\in[0,1]}\right)\equiv \chi(U)\; (\mathrm{mod}\, 2)$ if $n$ is odd.  
\end{Lemma}

\begin{proof}
It is straightforward to construct a family of functions $\phi_t(x)$, $t\in[0,1]$, with the following properties: 
\begin{enumerate}
\item  $\phi_0(x)=0$ and $\phi_1(x)=\phi(x) $;
\item the functions  are monotonically increasing in $t$; 
\item  for each critical point $x$ of $\wt\phi|_U=(\phi-2h)|_U$, there is a unique $t\in[0,1]$ such that $x$ is a critical point of  the function $\phi_t-2h$, of the same index and of critical value $0$.
\end{enumerate}
We associate with $\phi_t$ the front $\Gamma_{\cu;t}$ obtained from $\Gamma_{\cu} = \Gamma_{\cu;0}$ by replacing the lower branch of $\Gamma_{\cu}$, which by definition is  the graph 
$z=-h(x)$, by the graph $z=\phi_t(x)-h(x)$. The  Legendrian regular homotopy $\Lambda_{\cu;t}$ determined by this family of fronts has transverse self-intersections which correspond to critical points of $\phi_t-2h$ of critical value $0$. One can show that when $n=2k$ the sign of each intersection point equals $(-1)^{\ind\, x+k-1}$, where $\ind\,x$ is the Morse index of the critical point $x$ of the function $(\phi-2h)|_U$, see \cite{CieEli-Stein}.      
\end{proof}

To transport the stabilization construction to our standard local model, let  
$\Lambda_0^{U}=F^{-1}(\Lambda_{\cu}^{U})$ and $\Lambda_{0;\,t}=F^{-1}(\Lambda_{\cu;\,t})$, where $\Lambda_{\cu;\,t}$, $t\in[0,1]$, is the regular Legendrian homotopy constructed in Lemma \ref{lm:stab}. Then $\Lambda_{0;\,t}$,  $t\in[0,1]$,  is a compactly supported Legendrian regular homotopy connecting $\Lambda_0$ to $\Lambda^U_0$.
  
Consider a Legendrian $(n-1)$-submanifold $\Lambda$ of a contact $(2n-1)$-manifold $Y$ and a point $p\in\Lambda$. Fix a neighborhood $\Omega\subset Y$ of $p$ and a contactomorphism 
$$
\Phi\colon(\Omega,\Lambda\cap \Omega)\to (\R^{2n-1}_{\rm st},\Lambda_0).
$$
Replacing $\Omega\cap\Lambda$ with $\Phi^{-1}(\Lambda_0^U)$ we get
a Legendrian embedding $\Lambda^U$ which  coincides with $\Lambda$ outside of $\Omega$, and replacing it with $\Phi^{-1}(\Lambda_{0;\,t})$ we get a Legendrian regular homotopy
$\Lambda_t$, $t\in[0,1]$ connecting $\Lambda$ to $\Lambda^{U}$   
with $\SI\left(\{\Lambda_t\}_{t\in[0,1]}\right)$ equal to the minimal number of critical points of a Morse function on $U$ which attains its minimum value on $\p U$,  and with $I\left(\{\Lambda_t\}_{t\in[0,1]}\right)= (-1)^{k-1}\chi(U)$ if $n=2k$  and $I\left(\{\Lambda_t\}_{t\in[0,1]}\right)\equiv \chi(U)\;\mathrm{mod}\,2$ if $n$ is odd.  
We say that $\Lambda^U$ is obtained from $U$ via \emph{$U$-stabilization} in $\Omega$.
The most important case for us will be when $U$ is the ball. 
We say in this case that $\Lambda^U$ is the \emph{stabilization} of $\Lambda$ in $\Omega$ or simply the stabilization of $\Lambda$.

Let $\xi$ denote the contact plane field on $Y$, and note that there is an induced field of (conformally) symplectic 2-forms on $\xi$. We say that two Legendrian embeddings $f_0,f_1\colon\Lambda\to Y$ are \emph{formally Legendrian isotopic} if there exists a smooth isotopy $f_t\colon\Lambda\to Y$ connecting $f_0$ to $f_1$ and a $2$-parametric family of injective homomorphisms $\Phi_{s,t}:T\Lambda\to TY$ such that $\Phi_{0,t}=df_t$ for all $t\in[0,1]$, $\Phi_{s,0}=df_0$ and $\Phi_{s,1}=df_1$ for all $s\in[0,1]$, and such that $\Phi_{1,t}$ is a Lagrangian homomorphism $T\Lambda\to\xi$ for all $t\in[0,1]$. We will need the following simple lemma, see \cite{Eliash-Stein,CieEli-Stein,Mur}.

\begin{Lemma}\label{lm:stab-formal}
Let $\Lambda\subset Y$ be a Legendrian submanifold  and $\Lambda^U$ its  $U$-stabilization. Then if the Euler characteristic $\chi(U)$ of $U$ satisfies $\chi(U)=0$, then $\Lambda$ and $\Lambda^U$ are formally Legendrian isotopic.  
\end{Lemma}

\subsection{Loose Legendrian submanifolds}
Let $Y$ be a contact $(2n-1)$-manifold, $n>2$. We continue using the notation from Section \ref{ssec:stablization}. A Legendrian  embedding $\Lambda\to Y$ of a connected manifold $\Lambda$ (which we sometimes simply call a \emph{Legendrian knot}) is called \emph{loose}
if it is isotopic to the stabilization of another Legendrian knot.  
We point out  that looseness depends on the ambient manifold. A loose Legendrian embedding $\Lambda$ into a contact manifold $Y$ need not be loose in a smaller neighborhood $Y'$, $\Lambda\subset Y'\subset Y$. 

Any Legendrian submanifold $\Lambda\subset Y$ can be made loose by  stabilizing it in arbitrarily small neighborhood of a point.
Moreover, it can be made loose even without changing its formal Legendrian isotopy class. Indeed, one can first stabilize it
and then $U$-stabilize for some $U$ with $\chi(U)=-1$.

The following $h$-principle for loose Legendrian knots in contact manifolds of dimension $2n-1>3$ is proved in \cite{Mur}:  
\begin{thm}[\cite{Mur}]\label{prop:Murphy} Any two loose  Legendrian embeddings which coincide outside a compact set and which can be connected by a formal compactly supported Legendrian isotopy can be connected by a genuine compactly supported Legendrian isotopy.  
\end{thm}

\begin{Remark} 
It is also shown in \cite{Mur} that the $U$-stabilization of a Legendrian knot is loose for any non-empty $U$. 
\end{Remark}

These results imply the following:  
\begin{cor}\label{cor:loose-stab-loose} Any loose Legendrian knot $\Lambda$ is a stabilization of some other loose Legendrian knot $\Lambda'$.
\end{cor}
\begin{proof}   
Let $U \subset \Int(\R^{n-1}_+)$ be a ball and $C \subset \Int(\R^{n-1}_+)$ be a domain with Euler characteristic $-1$ disjoint from $U$. If $\Lambda$ is loose then according to Lemma \ref{lm:stab-formal} the stabilization $\Lambda^{C\cup U}=(\Lambda^C)^U$ is formally Legendrian isotopic to $\Lambda$. Hence in view of Theorem \ref{prop:Murphy} there is a genuine Legendrian isotopy connecting the stabilization $(\Lambda^{C})^{U}$ of the loose Legendrian knot $\Lambda'=\Lambda^C$ to $\Lambda$.
\end{proof}

\section{Lagrangian immersions with a conical singular point}\label{sec:conical}
In this section we establish an $h$-principle for maps which are self-transverse   Lagrangian immersions with the   minimal possible number of double points away from a single conical singularity. This result is a generalization and a corollary of the  corresponding result for Lagrangian embeddings from \cite{EliMur} which we state as  Theorem \ref{thm:caps-embedded} below.

\subsection{Legendrian isotopy and Lagrangian concordance}
The following result about realizing a Legendrian isotopy as a Lagrangian embedding of a cylinder is proved in \cite[Lemma 4.2.5]{ELIGRO-findim}. Let $Y$ be a contact manifold with contact structure $\xi$ given by the contact $1$-form $\alpha$, $\xi=\ker(\alpha)$. The symplectization of $Y$ is the manifold $\R\times Y$ with symplectic form $d(e^{s}\alpha)$, where $s$ is a coordinate along the $\R$-factor. 

Let $f_t\colon\Lambda\to Y$, $t\in[-1,1]$, be a Legendrian isotopy that is constant near its endpoints and which connects Legendrian embeddings $f_{-}=f_{-1}$ and $f_{+}=f_{1}$. We extend $f_t$ to all $t\in\R$ by setting $f_t=f_-$ for $t\leq -1$ and $f_t=f_+$ for $t\geq 1$. 
\begin{Lemma}\label{lm:Leg-Lag}
There exists a Lagrangian embedding
$$
G\colon\R\times\Lambda\to \R\times Y,
$$ 
given by the formula $G(t,x)=(h(x,t),\wt f_{t}(x))$ and with the following properties:
\begin{itemize}
\item there exists $T>0$ such that $G(t,x)=(t,f_-(x))$ for $t<-T$ and $G(t,x)=(t,f_+(x))$ for $t>T$;
\item $\wt f_{t}$ is a Legendrian embedding $C^\infty$-close to  $f_{t}$.
\end{itemize}
 \end{Lemma}

We will need the following modification of Lemma \ref{lm:Leg-Lag} for Lagrangian immersions.
Let $f_t\colon\Lambda\to Y$, $t\in[-1,1]$, be a self-transverse regular Legendrian homotopy constant near its endpoints that connects Legendrian embeddings $f_-$ and $f_+$. As in Lemma \ref{lm:Leg-Lag} we extend $f_t$ to all $t\in\R $ as independent of $t$ for $|t|\ge 1$.

\begin{Lemma}\label{lm:Leg-Lag-imm}
There exists a self-transverse Lagrangian immersion
$$
G\colon \R\times\Lambda \to \R\times Y,
$$ 
given by the formula $G(t,x)=(h(x,t),\wt f_{t}(x))$ and with the following properties:
\begin{itemize}
\item
there exists $T>0$ such that $G(t,x)=(t,f_{-}(x))$ for $t<-T$ and $G(t,x)=(t,f_{+}(x))$ for $t>T$;
\item $\wt f_{t}$ is  $C^\infty$-close to  $f_{t}$;
\item the double  points of $G$ are in one-to-one index preserving correspondence with the  double points of the regular homotopy $f_t$.
\end{itemize}
\end{Lemma}

\begin{proof}
The construction from \cite[Lemma 4.2.5]{ELIGRO-findim} which proves Lemma \ref{lm:Leg-Lag} can be applied with some additional care near the self-intersection instances of the Legendrian regular homotopy to prove Lemma \ref{lm:Leg-Lag-imm}. Here, however,  we will use a different argument.

Note that it is sufficient to consider the case when there is exactly one transverse self-intersection  point $q\in Y$ of the regular homotopy $f_t$ at the  moment $t=0 $. It is also sufficient to construct the immersed Lagrangian cylinder that corresponds to the Legendrian regular homotopy restricted to some interval $[ -\eps, \eps]$ for a small $\eps>0$, because then  one can apply Lemma \ref{lm:Leg-Lag} for the isotopy $f_t$ restricted to the rest of $\R$.

There exist local coordinates $(x,y,z)\in\R^{n-1}\times\R^{n-1}\times\R=\R^{2n-1}_{\st}$  in a neighborhood  $\Omega$ of $q$ such that the two intersecting branches $B_0$ and $B_1$ of $f_0$ at $q$ are given by the inclusion of 
\begin{align}
B_0&=\{(x,y,z)\colon y=0, z=0, |x|< 1\},\\
B_1&=\{(x,y,z)\colon x=0, z=0, |y|< 1\}.
\end{align}   
into $\R^{2n-1}_{\st}$. 

Modifying the regular homotopy $f_t$ slightly for $t$ close to $0$ we obtain a regular homotopy $\wt f_t$ without self-intersections for $t\in[0, \eps]$ and which, for $t\in[-\eps,0]$, is supported in $\Omega$ and has the following special properties:
\begin{itemize}
\item $\wt f_t \equiv \wt f_0 $ on $B_1$ and near $\p B_0$;
\item $\wt f_{-\eps}|_{B_0}$ is given by the
formula 
$$
\wt f_{-\eps}(\xi)=(x(\xi),y(\xi),z(\xi))=(\xi,0,-\delta),\quad \xi\in\R^{n-1},\;|\xi|\le1
$$
and, using the same notation, $\wt f_{0}|_{B_0}$ is given by the formula 
$$
\xi\mapsto\left(\xi,\frac{\p\phi}{\p \xi },\phi(\xi)\right),
$$ 
where  $\phi(\xi)=\delta\theta(|\xi|)$ for a small positive constant $\delta\ll\eps$ and a  $C^\infty$-function  $\theta\colon [0,1]\to[-1,1]$, which is equal to $-1$ near $1$, is equal to $1$ near $0$, and which has non-positive derivative.
\end{itemize}

The isotopy $\{\wt f_t\}_{t\in[0,\eps]}$ can be lifted to a Lagrangian cylinder in the symplectization using  Lemma \ref{lm:Leg-Lag}. Consider a $C^\infty$-function $\sigma\colon [-\eps,0]\times B_0\to\R$ with the following properties (using coordinates $(\tau,\xi)\in[-\eps,0]\times\R^{n-1}$, $|\xi|<1$ in analogy with the above):
\begin{itemize}
\item $\sigma(\tau,\xi)=-e^{\tau}\delta$  near $(\{-\eps\}\times B_0\times)\cup ([-\eps,0]\times\p B_0)$;
 \item $\sigma(\tau,\xi)=e^{\tau}\phi(\xi)$  near $0\times B_0$;
 \item  the function $\sigma(\tau,0)$, $\tau\in[-\eps,0]$ has a unique zero  in
 $(-\eps,0)$, which is moreover a regular value.
\end{itemize}
Now the required Lagrangian immersion $G|_{[-\eps,0]\times(B_0\cup B_1)}$ lifting
$\{\wt f_t\}_{t\in[-\eps,0]}$  to the symplectization $\R\times \R^{2n-1}_\st$ can be defined by the formula

\begin{alignat*}{2}
G(\tau,\eta)&=(t(\tau,\eta),x(\tau,\eta),y(\tau,\eta),z(\tau,\eta))&\\
&=(\tau,0,\eta,0),&\text{ if }(0,\eta,0)\in B_1,\\
G(\tau,\xi)&=(t(\tau,\xi),x(\tau,\xi),y(\tau,\xi),z(\tau,\xi))&\\
&=
\left(\tau,\xi,e^{-\tau}\frac{\pa\sigma}{\pa\xi}(\tau,\xi),e^{-\tau}\frac{\pa\sigma}{\pa \tau}(\tau,\xi)\right),&\text{ if }(\xi,0,0)\in B_0.\\
\end{alignat*}

Note that $G|_{B_0}$ and $G|_{B_1}$ are Lagrangian embeddings with respect to the symplectic form $\omega=d(e^t(dz-ydx))$. This is obvious for $G|_{B_1}$, for $G|_{B_0}$ we calculate 
\begin{align*}
 &G^*(d(e^t(dz-ydx))=d\left(e^\tau d\left(e^{-\tau}\frac{\pa\sigma}{\pa \tau}\right)- \frac{\pa\sigma}{\pa\xi}d\xi\right)=
d\left(-\frac{\pa\sigma}{\pa\tau}d\tau-\frac{\pa\sigma}{\pa\xi}d\xi\right)=0.
\end{align*}
We also have
$G(-\eps,\xi)=(-\eps,\xi,0,-\delta)=(-\eps,\wt f_{-\eps})$ and $G(0,\xi)=
\left(0,\xi,\frac{\p\phi}{\p \xi},\phi\right)=(0,\wt f_{0}).$
On the other hand,    the last property of the function $\sigma$ guarantees that $G(B_0)$ and $G(B_1)$ intersect transversely at a unique point and that  the index of intersection is the same as the self-intersection index of the regular homotopy $f_t$.
\end{proof}
           
\subsection{Lagrangian immersions with a conical point}
 Let $S^{2n-1}_\st=(S^{2n-1},\xi_\st)$  be  the  sphere with the standard contact structure  $\xi_\st$ defined by the restriction 
 $\alpha_\st$  to the unit sphere $S^{2n-1}\subset\R^{2n}$ of   the Liouville form $\lambda_\st:=\frac12\sum_{j=1}^n(x_jdy_j-y_jdx_j)$ in $\R^{2n}_\st$.
 
For any integer $m$ we denote by 
$(r,x)\in (0,\infty)\times S^{m-1}$   polar coordinates in
$\R^{m}-\{0\}$, i.e.~$x$ is the radial projection of a point
to the unit sphere, and $r$ is its distance to the origin.  The 
symplectic form $\omega_\st$ in $\R^{2n}$ has the form $d(r^2\alpha_\st)$ in polar coordinates.

A map $h\colon\R^n\to\R^{2n}_\st$ is called a \emph{Lagrangian cone} if $h^{-1}(0)=0$ and   if it is given by the formula
$h(r,x)=(cr^2,\phi(x))$ in polar coordinates, where $\phi\colon S^{n-1}\to S^{2n-1}_\st$ is a Legendrian embedding and $c$ is a positive constant.
 
Note that there exists a symplectomorphism
\begin{equation}\label{eq:symlectization-cone}
 C\colon (\R\times S^{2n-1}, d(e^t\alpha_\st))\to (\R^{2n}-\{0\},\om_\st)
\end{equation}
given by the formula
 $C(t,x)=\left(e^{\frac t2},x\right)$ in polar coordinates. Under this symplectomorphism
 Lagrangian cones in $\R^{2n}_\st$ correspond to cylindrical Lagrangian manifolds in the symplectization $(\R\times S^{2n-1}, d(e^t\alpha_\st))$ of $S^{2n-1}_\st$.

Let $L$ be an $n$-dimensional manifold and $X$ a $2n$-dimensional symplectic manifold.
A map $f\colon L\to X$ is called a \emph{Lagrangian immersion with a conical point} at $p\in L$, if $f|_{L-\{p\}}$ is a Lagrangian immersion, and if in a neighborhood of $p$ and in a Darboux chart  around $f(p)$, the map is equivalent to a Lagrangian cone $h\colon\R^n\to \R^{2n}_{\st}$ near the origin. The Legendrian   embedding   $\phi\colon S^{n-1}\to S^{2n-1}_\st$ corresponding to this cone is called  the \emph{link} of the conical point.

We define a regular Lagrangian homotopy $f_t\colon L\to X$, $t\in[0,1]$, of immersions with a conical point at $p$ to be a homotopy which is  fixed in some neighborhood of the singular point $p$ and that is an ordinary regular Lagrangian homotopy when restricted to $L-\{p\}$.  
For a self-transverse immersion $f$ with a conical point at $p$, we define the self-intersection    number $I(f)$ as $I(f|_{L\setminus\{p\}})$. Then $I(f)$ is invariant under regular homotopies fixed near $p$. 

 {Let  $g\colon\R^n\to\R^{2n}_\st$ be a Lagrangian immersion with a conical singularity at the origin $0\in\R^n$, and  which also coincides with a Lagrangian cone over a Legendrian link $\phi\colon S^{n-1}\to S^{2n-1}_\st$ outside a compact set.
Note that given a path  $\gamma\colon[0,1]\to\R^n$ connecting a point $\gamma(0)\in\R^n$ near infinity (i.e.~where the immersion is conical) and the origin $\gamma(1)=0\in\R^n$, the integral 
$\int_{g\circ\gamma}\lambda_\st$ is independent of the choice of $\gamma$. We will call it the {\em action of the singularity $0$ with respect to infinity}, and denote it by $a(g,0|\infty)$. 
Let $g_t\colon\R^n\to\R^{2n}_\st$, with $g_0 = g$, denote a Lagrangian regular homotopy,  compactly supported away from $0$.  Then $g_t$ is Hamiltonian if and only if $a(g_t,0|\infty)=\const$.
\begin{Lemma} \label{lm:compensation}
For any $\eps>0$, any smooth real-valued function $c\colon[0,1]\to\R$, $c(0)=0$, and any  Lagrangian cone  $h\colon\R^n\to\R^{2n}_\st$  over a Legendrian link $\phi\colon S^{n-1}\to S^{2n-1}_\st$, there exists a Lagrangian isotopy
$h_t\colon\R^n\to \R^{2n}_\st$ beginning at $h_0=h$, fixed near the singularity and outside the ball of radius $\eps$, and such that $a(h_t,0|\infty)=c(t)$, $t\in[0,1]$.
\end{Lemma}
\begin{proof}
Let $R^s\colon S^{2n-1}\to S^{2n-1}$, $s\in\R$, denote the time $s$ Reeb flow of the contact form $\alpha_\st$. Fix a non-positive $C^\infty$-function $\beta\colon\R\to\R$  with the following properties:
\begin{itemize}
\item  $\beta(s)=0$ for $s\notin(\frac1e,e)$;
\item $\int\limits_{ 1/e}^e \beta(u)du=-1$.
\end{itemize}
 Given $T,E\in\R$ let  
$$
g_{T,E}\colon\R\times S^{n-1}\to 
(\R\times S^{2n-1}_\st,d(e^t\alpha_\st))
$$ 
be the Lagrangian embedding given by the formula
$$
g_{T,E}(s,x)=(s-E,R^{T\beta(e^s)}(\phi(x))).
$$
Then, for $x_0\in S^{n-1}$,
\begin{align*}
\int\limits_{\R\times x_0}\left(g_{T,E}\right)^*(e^t\alpha_\st) &= Te^{-E}\int\limits_{-\infty}^\infty e^{2s} \beta'(e^s)ds = Te^{-E}\int\limits_{-1}^1 e^{2s} \beta'(e^s)ds\\ 
&=
Te^{-E}\int\limits_{\frac1e}^e u \beta'(u)du=-Te^{-E}\int\limits_{\frac1e}^e\beta(u)du=Te^{-E}.
\end{align*}
Then the required Lagrangian isotopy $h_t\colon\R^n\to \R^{2n}_\st$ 
with a conical singularity at the origin can be defined as
$$
h_t(x)=
C\circ g_{ {c(t)}/\eps,E},
$$
for $E=-2\log\eps$, $t\in[0,1]$, and $x\in\Lambda$.
\end{proof}
}

We will need the following result. 
\begin{Lemma}\label{lm:near-cone}
Let $h\colon\R^n\to\R^{2n}_\st$ be a Lagrangian cone over a Legendrian link $\phi\colon S^{n-1}\to S^{2n-1}_\st$. Then there exists a Hamiltonian regular homotopy $h_t\colon\R^n\to\R^{2n}_\st$, $t\in[0,1]$, with $h_0=h$,  which is fixed near the singularity and outside of a ball $B_R\subset\R^{n}$ of some radius $R>0$ centered at $0$, and such that the following hold:
\begin{itemize}
 \item $h_1$ coincides with the cone over a loose Legendrian knot $\wt\phi$  near $\p B_{\frac R2}$;
 \item the immersion $ h_1^-:=h_1|_{h_1^{-1}( B_{R/2}) }$
 has exactly one transverse self-intersection point;
 \item if $n=2k$ then, for any $\phi$, we can arrange that $I(h_1^-) =(-1)^{k-1}$, and if in addition $\phi$ is assumed to be a loose Legendrian knot, then we can arrange also that 
$I(h_1^-)=(-1)^{k}$.
   \end{itemize}
\end{Lemma}

\begin{Remark} 
Note that by scaling we can make the radius of the ball $R$ arbitrarily small.
\end{Remark}

\begin{proof}
Let $\phi\colon S^{n-1}\to S^{2n-1}_\st$ be the Legendrian link of the conical point. Let us denote by $\phi_1$ its stabilization. Note that $\phi_1$ is a  loose knot. If $\phi$ is itself loose then according to Corollary \ref{cor:loose-stab-loose} there exists   another  loose Legendrian knot
$\phi_{-1}$ such that $\phi$ is the stabilization of $\phi_{-1}$. We will call $\phi_{-1}$ the \emph{destabilization} of $\phi$.
According to  Lemma \ref{lm:stab} the embeddings $\phi_0:=\phi$ and $\phi_1$ can be included into   a regular Legendrian homotopy
$\phi_t\colon S^{n-1}\to S^{2n-1}_\st$, $t\in[0,1]$, such that  there is exactly one self-intersection point  for
$t\in(0,1)$, and  when $n=2k$ we have
$I(\{\phi_t\}_{t\in[0,1]})=(-1)^{k-1}$.
Similarly, if $\phi$ is loose then $\phi_{-1}$ and $\phi_0=\phi$
can be included into   a regular Legendrian homotopy
$\phi_t\colon S^{n-1}\to S^{2n-1}_\st$, $t\in[-1,0]$, with  one transverse  self-intersection point  which for an even    $n=2k$ has index
$(-1)^{k-1}$.

Next, we use Lemma \ref{lm:Leg-Lag-imm} to lift the 
Legendrian regular homotopy
$\{\phi_t\}_{t\in[0,1]}$ and its inverse $\{\phi_{1-t}\}_{t\in[0,1]}$, and in the loose case  $\{\phi_t\}_{t\in[-1,0]}$ and its   inverse 
$  \{\phi_{-t}\}_{t\in[0,1]}$, respectively, to
Lagrangian immersions
$$
G_1,G_2,G_3,G_4\colon \R\times S^{n-1}\to (\R\times S^{2n-1}, d(e^t\alpha_\st))
$$ 
with the following properties
\begin{itemize}
\item for a sufficiently large positive $t$ we have
\begin{align*}
&G_1(-t,x)=(-t,\phi(x)), \quad\quad G_1(t,x)=(t,\phi_1(x)),\\
&G_2(-t,x)=(-t,\phi_1(x)), \quad\;\; G_2(t,x)=(t,\phi(x)),\\
&G_3(-t,x)=(-t,\phi_{-1}(x)), \quad G_3(t,x)=(t,\phi(x)),\\
&G_4(-t,x)=(-t,\phi(x)), \quad\quad G_4(t,x)=(t,\phi_{-1}(x));
\end{align*}
\item each of these Lagrangian immersions has exactly 1 transverse self-intersection point;
\item if $n=2k$ then 
$$
I(G_1)=I(G_3)=(-1)^{k-1}, \quad
I(G_2)=I(G_4)=(-1)^k.
$$
\end{itemize}

Composing these immersions with the symplectomorphism $C^{-1}$ from
 \eqref{eq:symlectization-cone}  (and compactifying  with a conical point) we get Lagrangian immersions
$$
H_1,H_2, H_3,H_4\colon\R^n\to\R^{2n}_\st
$$
with a conical singularity at the origin. With appropriate rescaling we can glue together the immersions
 $H_1$ and $H_2$  to get  an immersion $H_{12}\colon\R^n\to\R^{2n}_\st$, and in the loose case glue $H_4$ and $H_3$ to get  an immersion $H_{43}\colon\R^n\to\R^{2n}_\st$, such that
 both are immersions with a  conical singularity with Legendrian link $\phi$, and both coincide with the cone over $\phi$ outside of the a ball $B_R$ of some radius $R>0$. In addition, near $\p B_{R/2}$  
 the immersion  $H_{12}$ coincides with the cone over $\phi_1$ and 
 $H_{43}$ coincides with the cone over $\phi_{-1}$. 
{  Using Lemma \ref{lm:compensation} we can modify the immersions $H_{12}$ and $H_{43}$ to arrange that $a(H_{12},0|\infty)= 
a(H_{43},0|\infty)=0$. Again with the help of Lemma \ref{lm:compensation}, we can therefore construct a 
 regular Hamiltonian homotopy $h_t$  connecting $h$ with $H_{12}$, which has   the required properties in the general case.} In the case  of loose $\phi$ and even $n=2k$, we can construct a Lagrangian  immersion with the index of its unique self-intersection point equal to $(-1)^k$, by taking  $h_t$ to be a  regular Hamiltonian homotopy connecting $h$ with $H_{43}$.
\end{proof}
 
\subsection{The $h$-principle for self-transverse Lagrangian immersions with a conical singularity and minimal self-intersection}\label{sec:con-h}
Let $X$ be a symplectic $2n$-manifold of dimension $2n>4$. The following $h$-principle for Lagrangian {\it embeddings} into $X$ with a conical point is proved in \cite{EliMur}.
\begin{thm}\label{thm:caps-embedded}
Let $f_0\colon L\to X$ be a Lagrangian immersion with a conical point $p\in L$ into a  simply connected symplectic manifold $X$ of dimension $2n>4$. If $n=3$ we further   assume that $X\setminus 
f_0(L)$ has infinite Gromov width, i.e. admits a symplectic embedding of an arbitrarily large ball.
If the Legendrian link of $f_0$ at $p$ is loose and if $I(f_0)=0$, then there exists a Hamiltonian regular homotopy $f_t\colon L\to X$, $t\in [0,1]$, that is fixed in a neighborhood of $p$ and that connects $f_0$ to a Lagrangian embedding $f_1$ with a conical point at $p$.
\end{thm}
 
As we shall see,  Theorem \ref{thm:caps-embedded} generalizes to self-transverse Lagrangian immersions with a conical point of non-zero Whitney index and with the minimal number of self-intersection points. In fact, Theorem \ref{thm:caps-embedded} itself is the key ingredient in the proof of its generalization, which we state next.
  
\begin{thm}\label{cor:caps-immersed}
Let $f_0\colon L\to X$ be a Lagrangian immersion with a conical point $p\in L$ into a  simply connected symplectic manifold $X$ of dimension $2n>4$.  If $n=3$ we further assume  that $X\setminus 
f_0(L)$ has infinite Gromov width.  If the Legendrian link of $f_0$ at $p$ is loose, then there exists a Hamiltonian  regular homotopy $f_t\colon L\to X$, $t\in [0,1]$, that is fixed in a neighborhood of $p$ and that connects $f_0$ to a self-transverse Lagrangian immersion $f_1$ with a conical point at $p$ and with $\SI(f_1)=|I({f_0})|$.
\end{thm}
  
\begin{proof}
We argue by induction on $d=|I({f_0})|$, using Theorem \ref{thm:caps-embedded} as the base of the induction for $d=0$.
Suppose the theorem holds for $|I({f_0})|=d-1$.
Let $\phi\colon S^{n-1}\to S^{2n-1}_\st$ be the loose  Legendrian link of the conical point and consider the immersion  $f_0\colon L\to X$. By definition, in some local Darboux neighborhood  near the singular point it is equivalent to a Lagrangian cone over $\phi$ in a ball $B_\eps$, $\eps>0$.
We use Lemma \ref{lm:near-cone} to construct a Hamiltonian regular homotopy supported in $B_\eps$ from $f_0$ to a new immersion $\wh f_0$ that coincides with the cone over  a loose knot  $\wt \phi$   near $\p B_{\eps/2}$ and which has exactly one transverse self-intersection point in $B_{\eps/2}$. If $n=2k $, we arrange  the index  to be of the same sign as $I(f_0)$.  
        
Let $\wt f_0$ be a Lagrangian immersion obtained from $\wh f_0$ by modifying it to the cone over $\wt\phi$ in $B_{\eps/2}$. We note that $|I(\wt f_0)|=d-1$, and by the induction hypothesis we find a Hamiltonian regular homotopy $\wt f_t$, $t\in[0,1]$,  fixed near the singularity and such that $\SI(\wt f_1)=|I(\wt f_0)|=d-1$.
 Note that the regular homotopy $\wt f_t$ is fixed in $B_\sigma$ for $\sigma\ll\frac\eps2$, but not necessarily fixed in $B_{\eps/2}$.
 
The required regular homotopy $f_t$ is then obtained by deforming $f_0$ into $\wh f_0$,
then scaling it inside $B_{\eps/2}$ to make it coincide with a  cone in $B_{\eps/2}- B_\sigma$, and finally deforming it outside $B_\sigma$, keeping it fixed in $B_\sigma$, using the Hamiltonian regular homotopy $\wt f_t$.
\end{proof}
 
\section{Proofs of the main results}

 \subsection{Proof of Theorem \ref{thm:min-intersect}}\label{ssec:proofmin-intersect}
Any   simple (i.e.~not double)  point $p\in f_0(L)$ can be viewed as a conical  point over a trivial Legendrian knot. Hence, we can apply
Lemma \ref{lm:near-cone} to find a Hamiltonian regular homotopy supported in a
Darboux neighborhood of $p$, symplectomorphic to $B_\eps$, such that the resulting Lagrangian immersion $\wt f_0$ coincides with a Lagrangian cone over a loose knot near $\p B_{\eps/2}$  and has  exactly one transverse self-intersection point in $B_{\eps/2}$, with index  equal to $(-1)^{k-1}$ if $n=2k$. Note that in all cases one then has $I(\wt f_0|_{L-\wt f_0^{-1}(B_{\eps/2})})=I(f_0)+(-1)^k$. Hence,  arguing as in the proof of Theorem \ref{cor:caps-immersed}, we replace $\wt f_0(L)\cap B_{\eps/2}$ by a cone over a loose Legendrian knot $\wt\phi$, and then  using Theorem \ref{cor:caps-immersed} we  construct a Hamiltonian regular homotopy $\wt f_t$ of the resulting immersion  $\wh f_0$, that is fixed in a neighborhood $B_\sigma$, $\sigma\ll\frac\eps2$, of the conical point, to an immersion with transverse self-intersections and with exactly $|I(f_0)+(-1)^{k}|$ double points. Finally 
    the required regular homotopy $f_t$ consists of first deforming $\wt f_0$ into $\wh f_0$, then scaling it inside $B_{\eps/2}$ in such a way that it becomes a cone in $B_{\eps/2}- B_\sigma$, and then deforming it outside $B_\sigma$ using the Hamiltonian regular homotopy $\wt f_t$ outside
 $B_\sigma$ and keeping it fixed in $B_\sigma$. The Lagrangian immersion $f_1$ is self-transverse and has exactly  
$$
|I(f_0)+(-1)^{\frac n2}|+1=
\begin{cases}
|I(f_0)|,&\text{if }(-1)^{\frac n2} I(f_0)<0,\\
|I(f_0)+2|,&\text{if } (-1)^{\frac n2}I(f_0)\geq0
\end{cases}
$$ 
double points.
\qed
 
\subsection{Proof of Corollary \ref{cor:immersion-Rn}}\label{ssec:proofimmersion-Rn}
Using Gromov's $h$-principle for Lagrangian immersions, see \cite{Gpdr}, we find an exact Lagrangian immersion  $f\colon L\to \R^{2n}_\st$ (in the given Lagrangian homotopy class $\sigma$ in cases (1) and (3)). Part (1) then follows from the corresponding case of Theorem \ref{thm:min-intersect}. 
If $n$ is even then $I({f_0})=(-1)^{\frac n2}\frac{\chi(L)}2$, see e.g.~\cite{EkEtSu}, [Proposition 3.2] and also part (3)  follows from the corresponding case of Theorem \ref{thm:min-intersect}.  To complete the proof of part (2), we observe that if $n=3$
then both smooth regular homotopy classes of immersions $f\colon L\to\R^{2n}_\st$, corresponding to $I({f})=0$ and $I({f})=1$, can be realized by a Lagrangian immersion.\qed

\subsection{Further results on $s(\sigma,L)$}\label{ssec:more-Rn}
If $n$ is even then the smooth regular homotopy class of a Lagrangian immersion $f\colon L\to\R^{2n}_\st$ is determined by $\chi(L)$, and thus Corollary \ref{cor:immersion-Rn} together with Gromov's non-existence result for exact Lagrangian embeddings gives complete information on $s(L)$ if $\chi(L)\le 0$. If $\chi(L)>0$, $s(L)$ depends on more intricate, differential topological, properties of $L$, see \cite{ES1,ES2} (Theorem \ref{thm:hmtpytype} gives information on the homotopy type of $L$ in this case). For odd $n$, $s(L)$ (and $s(L,\sigma)$) is determined by which of the two smooth regular homotopy classes contain Lagrangian (or equivalently totally real) immersions. The following result gives a partial answer. Recall that $\chi_2(L) = \sum_{j=0}^{\frac{n-1}2}\rk (H_j(L)) \mod 2$.

\begin{Theorem}     
If $L$ is an $n$-dimensional orientable closed manifold with $TL\otimes\C$ trivial, $n$ odd and $n\ge 3$, then the following hold:
\begin{enumerate}
\item If $n=3$ then both regular homotopy classes contain exact Lagrangian immersions.
\item If $n\ne 2^{q}-1$, $q\ne 2,3$ then only one regular homotopy class contains Lagrangian immersions. 
\item If $n=4k+1$ and $4k$ is not a power of two, or if the Stiefel-Whitney number  $w_2(L)w_{n-2}(L)$ vanishes, then only the regular homotopy class with Whitney index $I_f=\chi_2(L)$ contains Lagrangian immersions.
\item If $V$ is almost parallelizable then only the regular homotopy class with $I_{f}=\chi_2(V)$ contains Lagrangian immersions.
\end{enumerate}
\end{Theorem}

\begin{proof}
Cases (1), (2), and (4) are proved in \cite{Audin} and case (3) is a consequence of \cite{AsaEcc}.
\end{proof}

\subsection{Proof of Theorem \ref{thm:hmtpytype}}\label{ssec:proofhmtpytype}
We control the homotopy type of exact immersions $f\colon L\to\R^{2n}_\st$, $n$ even,  with exactly $\frac{1}{2}\chi(V)>0$ double points using a straightforward generalization of \cite[Lemma 2.2]{ES2}. We refer to \cite[Section 2]{ES2} for background and notation for the parts of (lifted) Legendrian homology that will be used below.  

Consider the Legendrian lift $\tilde f\colon L\to \R^{2n}_\st\times\R$. The Reeb chords of $\tilde f$ correspond to the double points of $f$, and the grading of all Reeb chords must be even since the sum of grading signs over Reeb chords equals $(-1)^{\frac{n}{2}}\chi(L)/2$, see \cite{EkEtSu}. Since no chord has odd grading it follows that the Legendrian algebra admits an augmentation, and since $L$ is spin we can use arbitrary coefficients in the Legendrian algebra.

Using the duality sequence for linearized Legendrian homology \cite{EESa} over $\Q$ we find that all odd dimensional homology of $L$ vanishes. In particular, the Maslov class vanishes and the linearized Legendrian homology admits an integer grading. The duality sequence with coefficients $\Z_{p}$ for arbitrary prime $p$ then implies that $L$ has only even dimensional homology over $\Z$.

We next claim that $\pi_1(L) = 1$. To this end, consider a connected covering space $\pi\colon \tilde L\to L$ of $L$. Then the ``lifted linearized Legendrian homology complex'' $\tilde{C}^{\mathrm{lin}}(\tilde L,\pi;k)$ with coefficients in the field $k$, see \cite[Section 2]{ES2}, has the form
\[
\tilde S\oplus \tilde C_{\mathrm{Morse}} \oplus \tilde L,
\]
where the elements in $\tilde S$ have grading $-1+2j$, $0\le j \le \frac{n}{2}$ and the elements in $\tilde L$ grading $2j+1$, $0\le j\le\frac{n}{2}$. Since $f(L)$ is Hamiltonian displaceable, the total homology of the complex vanishes. Since $\tilde{L}$ is connected, this in turn implies $\rk(\tilde S_{-1})=1$, where $\tilde S_{r}$ is the degree $r$ summand of $\tilde S$. Thus the covering $\pi$ has degree 1, and since the covering was arbitrary, that implies that $\pi_1(V) = 1$, as required.

Finally, if $\dim L > 4$ then we can use the Whitney trick to cancel homologically inessential handles.
\qed

\appendix
\section{Explicit constructions}\label{Sec:explicit}
In this section we consider explicit constructions of  Lagrangian immersions and Lagrangian regular homotopies. In Sections \ref{ssec:symplcoord} and \ref{ssec:frontslice} we introduce notation and some background material, which are necessary for the construction of a concrete Lagrangian immersion of $P=(S^{1}\times S^{n-1})\#(S^{1}\times S^{n-1})$ into $\R^{2n}_{\st}$. That construction, which is broken down into five stages, is given in the subsequent sections.  The construction generalizes  to dimensions $n\geq 3$ Sauvaget's construction from \cite{Sauvaget} for $n=2$.

\subsection{Symplectization coordinates}\label{ssec:symplcoord}
Consider $\R^{2n}_\st$ with coordinates $(x_1,y_1,\dots, x_n,y_n)$ and standard exact symplectic form $\omega = -d\beta$ with primitive $\beta=\sum_{j=1}^{n} y_j dx_j$. Let $\C^{n-1}\subset \R^{2n}_\st$ denote the subspace given by the equation $(x_n,y_n)=(0,0)$ and let $\beta_0=\beta|_{\C^{n-1}}$. Consider the contact manifold $\R\times\C^{n-1}$ with contact 1-form $\alpha=dz-\beta_0$, where $z$ is a coordinate in the additional $\R$-factor, and with symplectization $\R\times\R\times\C^{n-1}$ with exact symplectic form $d(e^{t}\alpha)$, where $t$ is the coordinate in the symplectization direction. Write $\xi=(x_1,\dots,x_{n-1})$ and $\eta=(y_1,\dots,y_{n-1})$, then $(\xi,\eta,x_n,y_n)$ are coordinates on $\R^{2n}_\st$. Consider the map $\Phi\colon \R\times\R\times\C^{n-1}\to\C^{n}$,
\[
\Phi(t,z,\xi,\eta)=(\xi,e^{t}\eta,e^t,z).
\]
Then
\[
\Phi^{\ast}(-\sum_{j}y_jdx_j)=-e^{t}(\eta\cdot d\xi +zdt)=e^{t}(dz-\eta\cdot d\xi)-d(e^{t}z)=\alpha-d(e^{t}z),
\] 
and hence $\Phi$ is an exact symplectomorphism from the symplectization $\R\times\R\times\C^{n-1}$ to $\C^{n}_{+}=\{(\xi,\eta,x_n,y_n)\colon x_n>0\}$. 

If $\Lambda\subset \R\times\C^{n-1}$ is a Legendrian submanifold then $\R\times\Lambda$ is an exact Lagrangian submanifold of the symplectization $\R\times\R\times\C^{n-1}$. If 
\[
(t,z(\lambda),\xi(\lambda),\eta(\lambda)),\quad \lambda\in\Lambda,\; t\in\R
\]
is a parameterization of $\R\times\Lambda$ then its image under $\Phi$ is parameterized by
\[
(\xi(\lambda),e^{t}\eta(\lambda),e^{t},z(\lambda))\in\C^{n}_{+}.
\]    
Conversely, if $L$ is a conical Lagrangian submanifold in $\C^{n-1}_{+}$ parameterized by
\[
(\xi(\lambda),s\eta(\lambda)),s,y_n(\lambda)),\quad \lambda\in\Lambda,\; s\in\R_{+},
\]
then the image of $L$ under $\Phi^{-1}$ is  the cylinder on a Legendrian submanifold $\Lambda\subset\R\times\C^{n-1}$ parameterized by $(z(\lambda),\xi(\lambda),\eta(\lambda))$, $\lambda\in\Lambda$, where $z(\lambda)=y_n(\lambda)$.

\subsection{Exact Lagrangian immersions by front slices}\label{ssec:frontslice}
Let $M$ be an $n$-manifold and let $f\colon M\to\C^{n}$ be an exact Lagrangian immersion. 
After small perturbation, the following general position properties hold:
\begin{enumerate}
\item All self-intersections of $f\colon M\to\C^{n}$ are transverse double points.
\item The coordinate function $x_n\circ f\colon M\to\R$ is a Morse function.
\item If $p=f(p_0)=f(p_1)$ is a double point of $f$ then $x_n(p)$ is a regular value of $x_{n}\circ f$. 
\end{enumerate}
Assume that $(1)-(3)$ hold. For any regular value $a$, the level set $M^{a}=(x_n\circ f)^{-1}(a)$ is a smooth $(n-1)$-manifold which is the boundary of the sublevel set $M^{\le a}=f^{-1}((-\infty,a])$, and if $\pi_n\colon \C^{n}\to \C^{n-1}$ is the projection that projects out $(x_n,y_n)$ then $f^{a}=\pi_n\circ f\colon M^{a}\to\C^{n-1}$ is an exact Lagrangian immersion. 

Double points of $f$ are also double points of $f^{a}$. In order to determine which double points of $f^{a}$ correspond to actual double points of $f$, we must recover the $y_{n}$-coordinate from $f^{a}$. Write $(x,y)=(\xi,\eta,x_n,y_n)$, where $\xi=(x_1,\dots,x_{n-1})$, $\eta=(y_{1},\dots,y_{n-1})$, and $(\xi,\eta)$ are standard coordinates on $\C^{n-1}$ as above. 
If $c$ is a double point of $f^{a}$, $c=f^{a}(c_0)=f^{a}(c_1)$ for $c_{0}\ne c_{1}\in M^{a}$ and if $\gamma$ is a path connecting $c_0$ to $c_1$ in $M^{a}$ then 
\[
y_n(c_1)-y_{n}(c_0)=\frac{d}{d x_n}\left(\int_{\gamma}\; \eta\, d\xi \right),
\]  
where $\eta\, d\xi=\sum_{j=1}^{n-1}y_j dx_j$. In other words, the $y_{n}$-coordinate at a double point is the $x_{n}$-derivative of the action of any path connecting its endpoints. Using the exactness of $f$, this can be rephrased as follows, if $z$ is a primitive of $f^{\ast}(\beta)$ and $z^{a}=z|_{M^{a}}$ then
\begin{equation}\label{eq:slicechord}
y_n(c_1)-y_{n}(c_0)=\frac{d}{d a}\left(z^{a}(c_1)-z^{a}(c_0)\right).
\end{equation}

In our constructions below we will depict the exact Lagrangian slices by drawing their fronts in $\R^{n-1}\times\R$. Before discussing slices, we give a brief general description of fronts. Let $N$ be a closed manifold and consider an exact Lagrangian immersion $g\colon N\to\C^{n}$. Pick a primitive $\zeta\colon N\to\R$ of $g^{\ast}(\beta)$. Then the map $G=(g,\zeta)\colon N\to\C^{n}\times\R$ is a Legendrian immersion, everywhere tangent to the contact plane field $\ker(d\zeta-\beta)$ on $\C^{n}\times\R$. The front of $g$ is the projection $\pi_{F}\circ G\colon N\to\R^{n}\times\R$, where $\pi_{F}$ projects out the $y$-coordinate. For generic $g$, the front has singularities; the front determines the original Lagrangian immersion via the equations
\[
y_j=\frac{\pa z}{\pa x_j},
\]  
which admit solutions that can be extended continuously over the singular set (caustic) of the front. For generic $g$, double points of $g$ lie off of the caustic and correspond to smooth points on the front with the same $x$-coordinate and with parallel tangent planes. The Reeb vector field of the contact form $dz-y\,dx$ on $\C^{n}\times\R$ is simply the coordinate vector field $\pa_{z}$ and thus double points of $g$ correspond to Reeb chords of $G$.

Below, we will study the fronts of exact Lagrangian immersions $f^{a}\colon M^{a}\to\C^{n-1}$ that are slices of a given Lagrangian immersion $f\colon M\to\C^{n}$, and as mentioned above it will be of importance to recover the $y_n$-coordinate at the double points of $f^{a}$. We will call the Reeb chords of the Legendrian lift of an exact slice $f^{a}$ \emph{slice Reeb chords}. Thus a slice Reeb chord is a vertical chord that connects two points on the front with parallel tangent planes. In our pictures below, we indicate the difference of the $y_n$-coordinate at the endpoints by showing whether the vertical chord that connects them grows or shrinks, cf.  Equation \eqref{eq:slicechord}.

\subsection{Overview: a  construction in five pieces}\label{ssec:onedouble}
The rest of the Appendix is devoted to the construction of a self-transverse exact Lagrangian immersion $P = (S^1 \times S^{n-1} \# S^1 \times S^{n-1}) \rightarrow \C^n$ with exactly one double point.  The construction of the immersion is broken down into five stages, two of which are described in terms of fronts, and three of which are described via Lagrangian slices.

In what follows, we  write $\xi=(x_1,\dots,x_{n-1})$ and $t=x_n$. Let $\eta=(y_1,\dots,y_{n-1})$ be dual to $\xi$ and $\tau=y_n$ dual to $t$. If $f(\cdot, t)$ is a function that depends on $t$, we write $f'(\cdot,t)$ for the partial derivative $\frac{\pa f}{\pa t}$.

The construction will be decomposed into the following five pieces: a bottom piece $(i)$ which will be drawn in Section \ref{ssec:1stpiece} as a sequence of fronts and which contains five critical slices (i.e.~slices containing critical points of the Morse function $x_n\colon P\to\R$), three middle pieces $(ii)-(iv)$, see Sections \ref{ssec:2ndpiece}, \ref{ssec:3rdpiece}, and \ref{ssec:4thpiece} all without critical slices where we will draw the corresponding Lagrangian immersions into $\C^{n}$, and finally a top piece $(v)$, also drawn by fronts, that contains the unique double point,  see Section \ref{ssec:5thpiece}.

\subsection{The first piece of the immersion}\label{ssec:1stpiece}
The first piece is constructed in four steps.
\subsubsection{}
We pass the minimum of the $t$-coordinate and a standard $(n-1)$-sphere appears as shown in Figure \ref{fig:bottom1}. It has one slice Reeb chord $c_1$ of length $\ell(c_1)$ and $\ell'(c_1)>0$.
\begin{figure}
\centering
\includegraphics[scale=0.4]{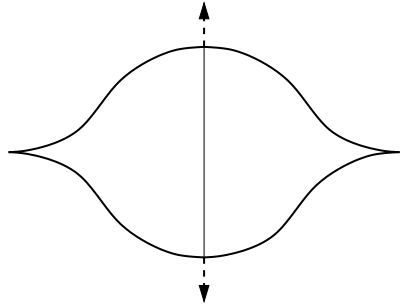} 
\caption{A standard sphere is born.}
\label{fig:bottom1}
\end{figure}

\subsubsection{}
We introduce two Bott families $\mathbf{c}_2$ and $\mathbf{c}_3$ of slice Reeb chords of index $0$ and $1$, respectively. Both families are topologically $(n-2)$-spheres, symmetric about the central slice Reeb chord. The lengths of the slice Reeb chords of the families are $\ell(\mathbf{c}_2)<\ell(\mathbf{c}_{3})$,  $\ell'(\mathbf{c}_{2})<0$, and $\ell'(\mathbf{c}_{3})<0$; see Figure \ref{fig:bottom2}. 
It will be important later that $\ell(\mathbf{c}_3)$ is not too small compared to $\ell(c_1)$. We introduce the following quantities corresponding to certain areas in the Lagrangian slice projections which appear later, but here related to the lengths of the slice Reeb chords:
\[
\ell(\mathbf{c}_{3})=C,\quad\ell(\mathbf{c}_2)=-B+C,\quad\ell(c_1)=A-B+C,
\] 
where $C>0$ and $A>B>0$. Then $C'<0$, $-B'+C'<0$, and $A'>B'-C'$. To be definite about $\ell(\mathbf{c}_3)$, we take $B=\frac{\mu_0}{5} C$ and $A=\frac{\mu_1}{2}C$, where $\mu_j$ are parameters such that $\mu_j\approx 1$. (Since the slice Reeb chords must change length with $t$, we cannot enforce $\mu_j=1$ but we can take these parameters to be approximately equal to $1$. ) 

\begin{figure}
\centering
\includegraphics[scale=0.45]{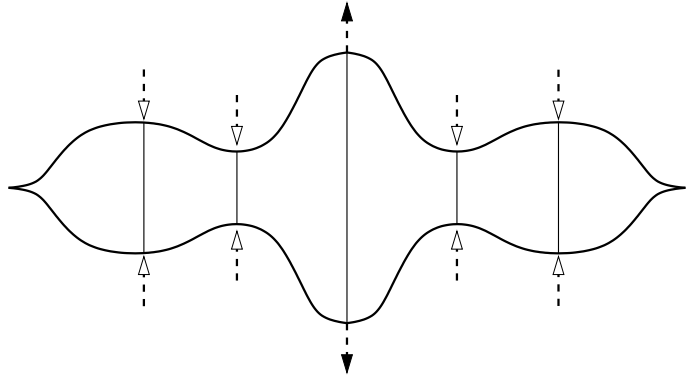} 
\caption{New born Bott families of shrinking chords.}
\label{fig:bottom2}
\end{figure}

\subsubsection{}
We Morse modify along the Bott family $\mathbf{c}_2$ of slice Reeb chords of minimal length twice, first  adding a family of 1-handles and then removing them, as shown in Figure \ref{fig:bottom3}. It is straightforward to check that the manifold that results from these modifications is a punctured connected sum $(S^{1}\times S^{n-1})\#(S^{1}\times S^{n-1})$, i.e., $P-D^{n}$.
The slice sphere that appears right after the second Morse modification is depicted in Figure \ref{fig:bottom4}.
\begin{figure}
\centering
\includegraphics[scale=0.3]{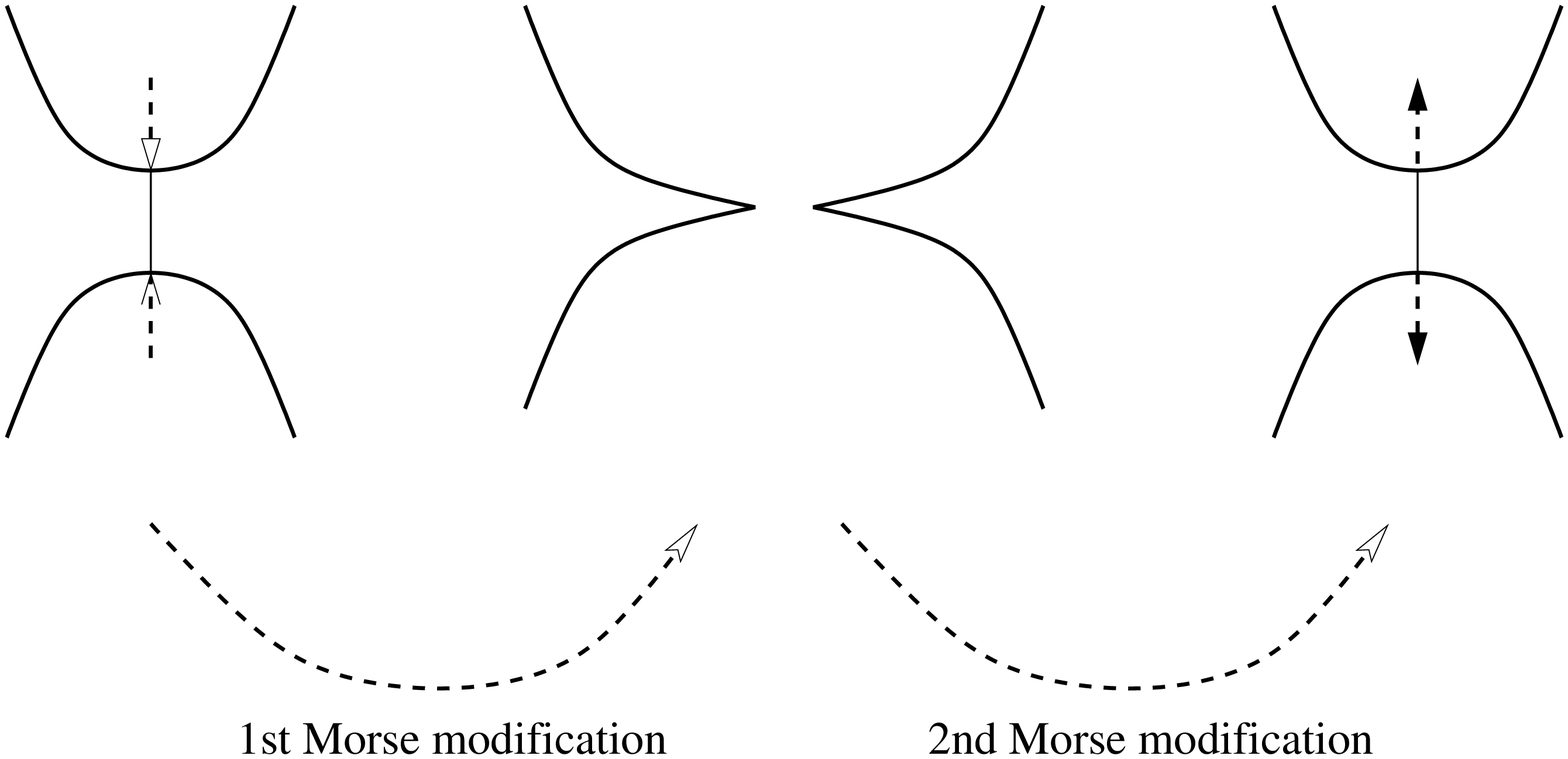} 
\caption{Morse modification along the Bott families of minima.}
\label{fig:bottom3}
\end{figure}

\subsubsection{}
In the final step of the first piece, we Morsify the Bott family $\mathbf{c}_2$ of slice Reeb chords. In doing so, as we shall see below, it is natural to think about the quantities $A$, $B$, and $C$ as functions of $(v,t)\in S^{n-2}\times\R$, where we think of $S^{n-2}$ as the unit sphere $\{\xi\in\R^{n-1}\colon |\xi|=1\}$, which naturally parameterizes the Bott manifold. We Morsify $\mathbf{c}_{2}$ leaving one very short chord $c_{2}^{0}$ with $\ell(c_{2}^{0})=\epsilon^{2}$ lying in direction $-v_0\in S^{n-2}$ and one long chord $c_{2}^{n-2}$ in direction $v_0 \in S^{n-2}$,  with 
\begin{equation}\label{eq:AandCclose}
\ell(c_{2}^{n-2})=-B(v_0,t)+C=(1-\frac{\mu_0}{5})C\approx\frac{4}{5}C,
\end{equation}
see Figure \ref{fig:bottom4}. Here, $\mu_0\approx 1$ is as above, and $C$ is constant in $v$ since we keep the Bott symmetry of $\mathbf{c}_3$.
The superscripts on $c_{2}^{0}$ and $c_{2}^{n-2}$ refer to the Morse indices of the slice Reeb chords, considered as positive function differences of the functions defining the two sheets of the front at their endpoints.  

Let $\phi = \phi_+ - \phi_-$ be such a positive difference between local functions defining the sheets of the front.  For future reference, we assume that the Morsification has only a small effect on the  level surfaces of $\phi$, which remain close to those of the Bott situation.  In particular,  the level curves of $\phi$ for values close to $\phi(c^{n-2}_{2})$ are everywhere transverse to the radial vectors along an $S^{n-2}$ slightly outside the former Bott-manifold, see Figure \ref{fig:levels}.  

\begin{figure}
\centering
\includegraphics[scale=0.4]{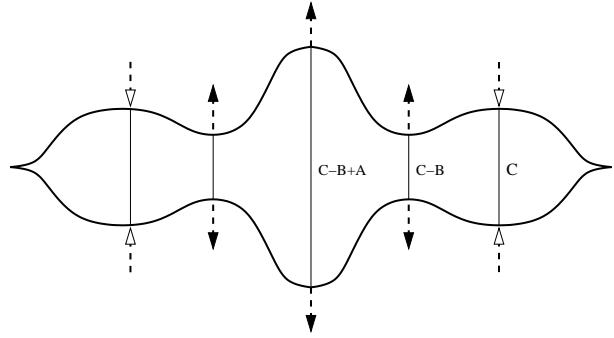} 
\caption{Final stage of the first piece.}
\label{fig:bottom4}
\end{figure}

\begin{figure}
\centering
\includegraphics[scale=0.5]{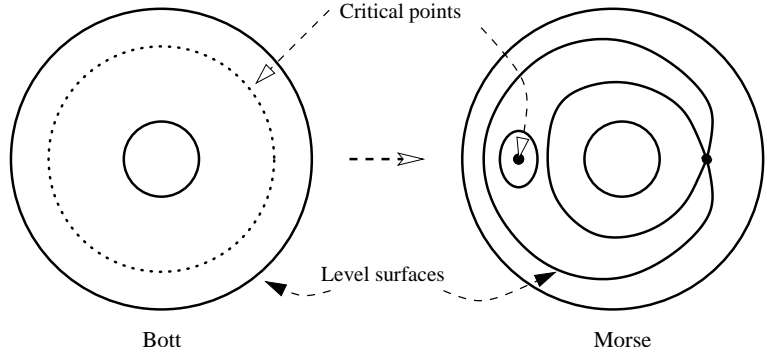} 
\caption{Level surfaces after Morsification.}
\label{fig:levels}
\end{figure}

\subsection{A guide to reading pictures of Lagrangian slices}
For the next three pieces of the immersion, we will not use the front representation as above, but will instead draw a family of exact Lagrangian slices. Before reaching the details we describe how to read the pictures.

Let $v\in S^{n-2}=\{\xi\in \R^{n-1}\colon |\xi|=1\}$ and let $r\in \R_{\ge 0}$. We construct our exact Lagrangian slice-spheres by drawing their slices in the $2$-dimensional half-planes determined by $v$ and its dual vector $w$, i.e. $w$ is a vector in $\eta$-space $\R^{n-1}$ dual to $v$ which lies in $\xi$-space. These slices are curves $\gamma_{v}\colon [0,1]\to\{rv+\rho w\colon r\ge 0, \rho\in\R\}$ that begin and end at a central point (the location of the slice Reeb chord of maximal length of the corresponding front) at $(0,0)=0\in\C^{n-1}$. In order for the curves $\gamma_{v}$ to close up and form a sphere, the integral $\int_{\gamma_{v}}\eta\, d\xi$ must be independent of $v\in S^{n-2}$. 

We will draw the immersed curves $\gamma_{v}$ with over/under information recording the value of the $\tau=y_n$-coordinate  at double points. It is also important to keep track of the values of $z$ at crossings, where $z$ is a primitive function of the exact Lagrangian slice. We write $(\Delta\tau)_j$ and $(\Delta z)_j$ for the differences in $\tau$-coordinate and $z$-coordinate, respectively, between the upper and the lower strands at the crossing labeled $j$ in figures. Note that with these conventions $(\Delta z)_j\ge 0$.

We next need a description of the actual double points of the slice immersions $f^{a}$ in terms of the curves $\gamma_{v}$. Note first that any double point of $f^{a}$ corresponds to a double point of some curve $\gamma_{v}$. Although we break the $S^{n-2}$ symmetry we stay fairly close to a symmetric situation. In particular, the double points of the curves $\gamma_{v}$ (with the exception of the central double point) will come in $S^{n-2}$-families and the $z$-differences $(\Delta z)_j$ then give functions
\[
(\Delta z)_j\colon S^{n-2}\to\R_{\geq 0}.
\] 
Since double points correspond to parallel tangent planes on the front we find that the double points of $f^{a}$ are exactly the critical points of these functions.

\subsection{The second piece of the immersion}\label{ssec:2ndpiece}
We will represent the second piece  of the immersion in three steps. 

\subsubsection{}
In the initial step the curves correspond to the front in Figure \ref{fig:bottom4}, which is depicted in the slice model in Figure \ref{fig:lag1}. 
 $A$, $B$, and $C$ denote the (positively oriented) areas indicated; each is a function of $v$ and $t$. We then have
\begin{alignat*}{2}
&(\Delta z)_1(t) = A-B+C, &\qquad &(\Delta\tau)_{1}=(A-B+C)'>0,\\
&(\Delta z)_2(v,t) = -B+C, &\qquad &(\Delta\tau)_{2}=(-B+C)'>0,\\
&(\Delta z)_3(v,t) = -C, &\qquad &(\Delta\tau)_{3}=(-C)'>0.
\end{alignat*}
 In particular $A(-v_0,t)>B(-v_0,t)=C(-v_0,t)-\epsilon^{2}$ and, for fixed $t$, the functions $A(v,t)$ and $B(v,t)$ have maxima at $-v_{0}$ and minima at $v_0$, whereas the function $C(v,t)$ is constant in $v$.  
\begin{figure}
\centering
\includegraphics[scale=0.5]{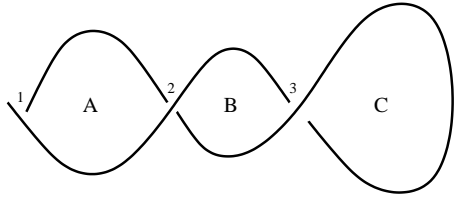} 
\caption{Initial Lagrangian slice.}
\label{fig:lag1}
\end{figure}

\subsubsection{}
We apply a finger move across the area $B(v,t)$, splitting it into two pieces $B_0(v,t)$ and $D(v,t)$ as shown in Figure \ref{fig:lag2}, where $B_0(v,t)$ is constant in $v$: 
\[
B_0(v,t)=B_0(v_0,t)=D(v_0,t).
\] 
Consequently, $D(v,t)$ equals $B(v,t)-\const$.

The crossing conditions then read:
\begin{alignat*}{2}
&(\Delta z)_1 = A-B_0+C-D, &\qquad &(\Delta\tau)_{1}=(A-B_0+C-D)'>0,\\
&(\Delta z)_2 = -B_0+C-D, &\qquad &(\Delta\tau)_{2}=(-B_0+C-D)'>0,\\
&(\Delta z)_3 = -C, &\qquad &(\Delta\tau)_{3}=(-C)'>0,\\
&(\Delta z)_4 = -D, &\qquad &(\Delta\tau)_{4}=(-D)'>0.
\end{alignat*}
\begin{figure}
\centering
\includegraphics[scale=0.6]{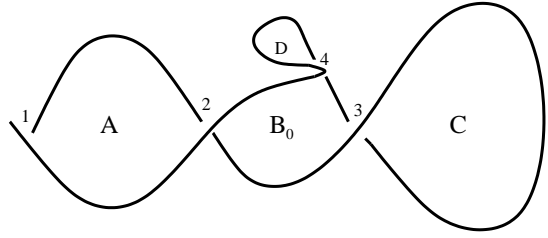} 
\caption{Applying a finger move.}
\label{fig:lag2}
\end{figure}

\subsubsection{}
We continue the finger move and introduce a new small area $E(v,t)$, as shown in Figure \ref{fig:lag3}. We have
\begin{alignat}{2}\label{eq:initial1}
&(\Delta z)_1 = A-B_0+C-D+E, &\qquad &(\Delta\tau)_{1}=(A-B_0+C-D+E)'>0,\\\label{eq:initial2}
&(\Delta z)_2 = -B_0+C-D+E, &\qquad &(\Delta\tau)_{2}=(-B_0+C-D+E)'>0,\\\label{eq:initial3}
&(\Delta z)_3 = -C, &\qquad &(\Delta\tau)_{3}=(-C)'>0,\\\label{eq:initial4}
&(\Delta z)_4 = -D+E, &\qquad &(\Delta\tau)_{4}=(-D+E)'>0,\\\label{eq:initial5}
&(\Delta z)_5 = -D, &\qquad &(\Delta\tau)_{5}=(-D)'>0.
\end{alignat}
\begin{figure}
\centering
\includegraphics[scale=0.6]{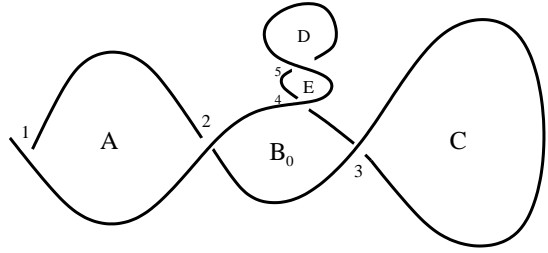} 
\caption{Final step of lower middle piece. The newly introduced area $E(v,t)$ is everywhere small at this stage.}
\label{fig:lag3}
\end{figure}

\subsection{The third piece of the immersion}\label{ssec:3rdpiece}
The slices of the third piece will be drawn in the same way as the slices in Section \ref{ssec:2ndpiece}. We will distinguish small and large deformations, and use $\approx_{t}$ for quantities that are almost constant in time, and $\approx_{(v,t)}$ for those that are almost constant in both $v\in S^{n-2}$ and in time. Recall that we wrote $\epsilon^{2}$ for the height of the smallest Reeb chord, where $0<\epsilon\ll 1$; by  a small deformation we mean one smaller than  $\epsilon^{p}$ for $p\gg 2$. For example $C(v,t)\approx_{(v,t)}\const$ means that the area $C(v,t)$ is independent of $v$ and that it has  $t$-derivative of order $\epsilon^p$, $p\gg 2$, with the sign of the derivative dictated by the crossing conditions. 

\subsubsection{}
In the initial phase of the third piece we change the area functions in such a way that the following conditions are met:
\[
A(v,t)+E(v,t)\approx_{(v,t)}\const,\; B_0(v,t)\approx_{(v,t)}\const,\; C(v,t)\approx_{(v,t)}\const,\; D(v,t)\approx_{t} D(v,t'), 
\]
where $t'$ is the starting time.
Consider $v_0\in S^{n-2}$ as above. At $-v_0$, the values of all area functions stay close to their initial values. 
At $v_0$, $A(v,t)$ shrinks toward $0$ and $E(v,t)$ grows correspondingly. We choose these functions so that they have exactly two critical points $\pm v_0$. Note that these deformations are compatible with \eqref{eq:initial1}--\eqref{eq:initial5}. At the point when $A(v_0,t)=0$ we find that the central slice Reeb chord which corresponds to a maximum of the function difference determined by the two sheets of the front and to the double point labeled $1$ cancels with the chord $c_{2}^{n-2}$ corresponding to the largest value of $(\Delta z)_2$. The curve $\gamma_{v_0}$ 
at the cancelling moment is depicted in Figure \ref{fig:noninv1}.
\begin{figure}
\centering
\includegraphics[scale=0.5]{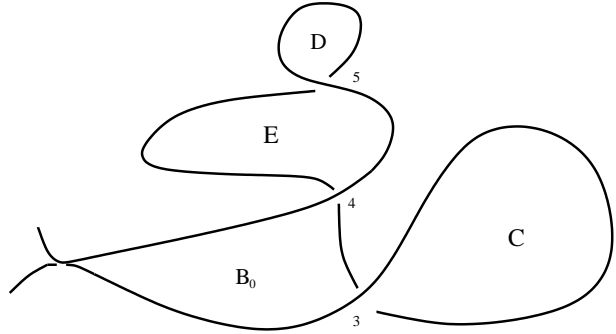} 
\caption{The cancelling moment.}
\label{fig:noninv1}
\end{figure}
Note that at this point the slice Reeb chord that corresponds to the double point labeled $4$ satisfies the following:
\begin{equation}\label{eq:rightsize}
(\Delta z)_{4}(v,t)<E(v_0,t)-D(v_0,t)\approx_{t} A(v_0,t') - D(v_0,t) < C(v,t),
\end{equation} 
where we recall that $C(v,t)\approx_{(v,t)}\const$ and that $t'$ denotes the initial instance, before we start shrinking $A(v,t)$, see \emph{$(i)$ Step 4}. In particular, the Reeb chord of maximal length is the chord labeled $3$ (denoted $\mathbf{c}_{3}$ earlier) of length $C(v,t)$. 

\subsubsection{}
In the central region where the Reeb chords cancel, the front of the Lagrangian slice consists of two function graphs. The level sets of the positive difference of these function graphs are shown in Figure \ref{fig:levelcancel}. 
\begin{figure}
\centering
\includegraphics[scale=0.4]{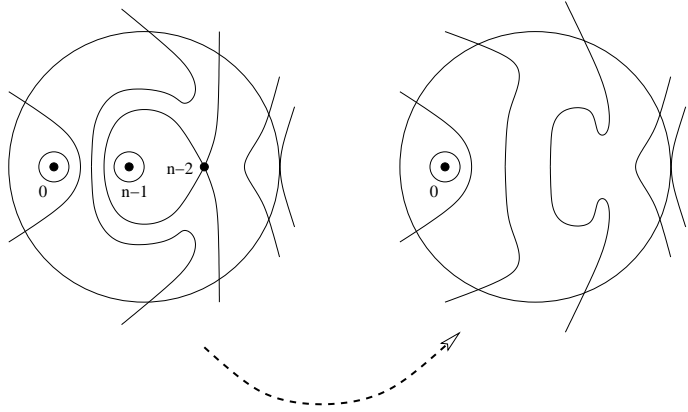} 
\caption{Level sets in the central region near the moment of cancellation.}
\label{fig:levelcancel}
\end{figure}

Recall the consequence of the earlier Bott set up, see \emph{$(i)$ Step 4}, that the level sets are transverse to the radial vectors along an $S^{n-2}$ surrounding the central region. Thus isotoping level sets keeping them fixed along the boundary, we see that there exists a Hamiltonian isotopy which is fixed outside the central region that deforms the Lagrangian so that the central region appears as shown in Figure \ref{fig:hamilton1}. Here the slice Reeb chord $c_{2}^{0}$ is the new central Reeb chord and level sets are everywhere transverse to the radial vector field. Noting that double points in the central region of the radial slices $\gamma_v$ correspond to tangencies of the level sets and the radial vector field, we find that after this Hamiltonian deformation, the curves $\gamma_{v}$ are as shown in Figure \ref{fig:hamilton2}. We take the slices of our exact Lagrangian to be approximately equal to the instances of this Hamiltonian deformation,  deviating from it slightly in order to ensure that the crossing conditions hold. 
\begin{figure}
\centering
\includegraphics[scale=0.4]{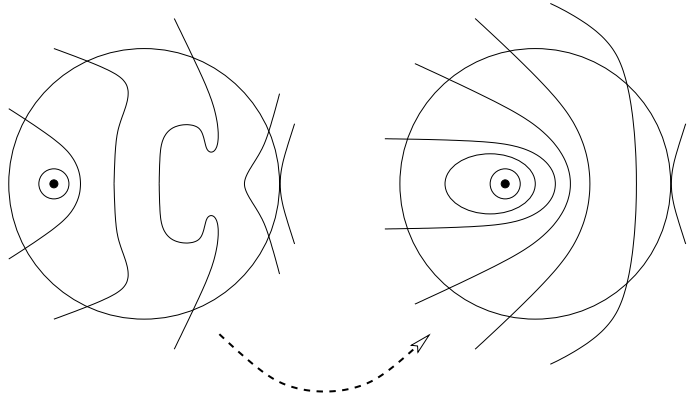} 
\caption{Placing the minimal length Reeb chord in central position by an (almost) Hamiltonian deformation.}
\label{fig:hamilton1}
\end{figure}

\begin{figure}
\centering
\includegraphics[scale=0.6]{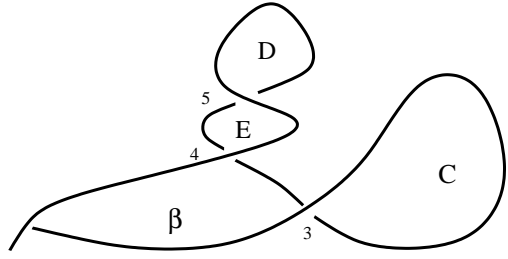} 
\caption{Appearance of the curves $\gamma_{v}$ after $c_{2}^{0}$ became the new central Reeb chord.}
\label{fig:hamilton2}
\end{figure}

We have the following crossing conditions:
\begin{alignat}{2}\label{eq:new1}
&\epsilon^{2} \approx -\beta+C-D+E, &\qquad &(\Delta\tau)_{0}=(-\beta+C-D+E)'>0,\\\label{eq:new2}
&(\Delta z)_3 = -C, &\qquad &(\Delta\tau)_{3}=(-C)'>0,\\\label{eq:new3}
&(\Delta z)_4 = -D+E, &\qquad &(\Delta\tau)_{4}=(-D+E)'>0,\\\label{eq:new4}
&(\Delta z)_5 = -D, &\qquad &(\Delta\tau)_{5}=(-D)'>0.
\end{alignat}
For the deformation described above, the function $\beta(v,t)$ in Figure \ref{fig:hamilton2} has only two critical points, with maximum value at $v_0$: 
\begin{equation} \label{maxbeta}
\beta(v_0,t)\approx A(-v_0,t) + B_0(v_0,t) > C(v,t),
\end{equation}
and minimum value at $-v_0$:
\begin{equation} \label{minbeta}
\beta(-v_0,t)\approx B_{0}(-v_0,t)< C(v,t).
\end{equation}

\subsection{The fourth piece of the immersion}\label{ssec:4thpiece}
We describe the fourth piece in four steps.

\subsubsection{} We rotate the ends of the curves $\gamma_{v}$ as shown in Figure \ref{fig:rot}. More precisely, recall that $(r,\rho)\mapsto rv+\rho w$ are coordinates on the half plane in which $\gamma_{v}$ lies.  Consider an interval  $[0,R]$ and $r_0\in[0,R]$ small, such that (i) all curves $\gamma_{v}$ are of standard form in $\{(r,\rho)\colon 0\le r\le 2r_0\}$, and (ii)  all curves $\gamma_{v}$ are contained in $\{(r,\rho)\colon 0\le r\le \frac12 R\}$. Consider a Hamiltonian deformation $\psi_{t}$ that is constant in $\{(r,\rho)\colon 0\le r\le r_0\}$, that is a $\frac{\pi}{2}$ rotation in the region $\{(r,\rho)\colon 2r_0\le r\le \frac12R\}$, and that is  again constant in $\{(r,\rho)\colon \frac34R\le r\}$. The curves $\gamma_{v}$ are then given by small deformations of the curves $\psi_{t}\circ\gamma_{v}$; the small deformations shrink the areas $C(v,t)$ and $D(v,t)$, and increase the area $\beta(v,t)$, in each case with very small derivative, in order to have the crossing conditions satisfied. 
\begin{figure}
\centering
\includegraphics[scale=0.5]{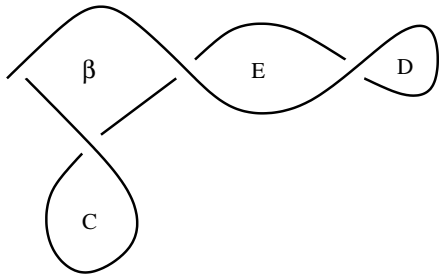} 
\caption{A rotation has been applied.}
\label{fig:rot}
\end{figure}

\subsubsection{} We shrink the area $C(v,t)$ for $v$ in a certain subset of $S^{n-2}$. Recall from Equations \eqref{maxbeta},\eqref{minbeta} that the minimum of the function $\beta(v,t)$ is $\beta(-v_0,t)\approx B_0(-v_0,t)=B(v_0,t)< C(v,t)$ whereas its maximum $\beta(v_0,t)\approx B(v_0,t)+A(-v_0,t)>C(v,t)$. Consider the subset $Q\subset S^{n-2}$ where 
\[
\beta(v,t)\ge C(v,t)-\frac12\epsilon. 
\]
Over $Q$ we shrink $C(v,t)$ and $\beta(v,t)$ in such a way that $C(v,t)-\beta(v,t)\approx \const$ until $C(v,t)=\epsilon$. Outside a neighborhood of $Q$ we keep $C(v,t)$ at its original size. Note that outside $Q$ we have
\begin{equation}\label{eq:EleD}
E(v,t)= D(v,t)+\beta(v,t)-C(v,t)+\epsilon^{2}\le D(v,t) -\tfrac12\epsilon +\epsilon^{2} < D(v,t).
\end{equation}

\subsubsection{} We slide $C(v,t)$ out across $E(v,t)$ without creating double points. To see that this is possible we subdivide into three cases. First, around $Q$, $C(v,t)$ is small (size $\epsilon$) compared to $E(v,t)$, which is of (order $1$) finite size. In this case there are two double points created when $C(v,t)$ enters and leaves $E(v,t)$, see Figure \ref{fig:Csmall}. The crossing conditions at the entering crossings read
\[
(\Delta\tau)_{\mathrm{en}}=(\tilde C(v,t)-D(v,t)+E(v,t))'>0,
\]
and those at the exiting crossings
\[
(\Delta\tau)_{\mathrm{ex}}=(\tilde C(v,t)-D(v,t))'>0,
\]
where $0<\tilde C(v,t)<C(v,t)=\epsilon$. Here we let $E(v,t)\approx\const$ and we let $D(v,t)$ shrink to keep the above derivative positive. As the total variation of $\tilde C(v,t)$ (i.e.~the total amount of area transported in and out of $E(v,t)$) is $2\epsilon$ and since $D(v,t)>2\epsilon$, $D(v,t)$ is sufficiently large to keep the derivative positive at all times.

\begin{figure}
\centering
\includegraphics[scale=0.6]{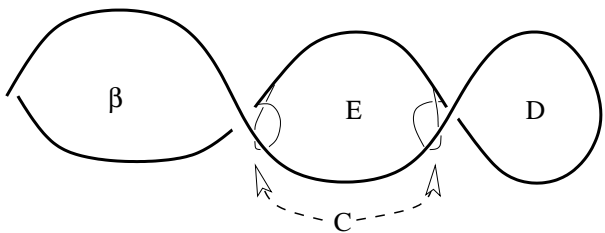} 
\caption{Sliding a small area $C(v,t)$.}
\label{fig:Csmall}
\end{figure}

Second, well outside a neighborhood of $Q$, $C(v,t)$ is large and there are two double points throughout the deformation, see Figure \ref{fig:Cbig}, where the crossing conditions read
\begin{align*}
(\Delta\tau)_{6}&=(\tilde C(v,t)-D(v,t)+\tilde E(v,t))'>0,\\
(\Delta\tau)_{7}&=(-D(v,t)+\tilde E(v,t))'>0,\\
\end{align*}
where $\tilde C(v,t)\approx C(v,t)-\const$ and where $0\le \tilde E(v,t)\le E(v,t)$. Here $\tilde E'(v,t)<0$ and we ensure that the crossing conditions hold by shrinking $D(v,t)$. Since the total variation of $\tilde E(v,t)$ is at most the area $E(v,t)$ and $D(v,t)> E(v,t)$ in this region, we find that the crossing condition is met at all times. 

Third, in the region where we interpolate between small and large $C(v,t)$, $E(v,t)$ is already smaller than $D(v,t)$, and the amount of area transported through $E(v,t)$ is smaller than the corresponding amount described above; hence  the necessary crossing conditions can be arranged by shrinking $D(v,t)$. In conclusion we can thus slide $C(v,t)$ out for all $v$.

\begin{figure}
\centering
\includegraphics[scale=0.6]{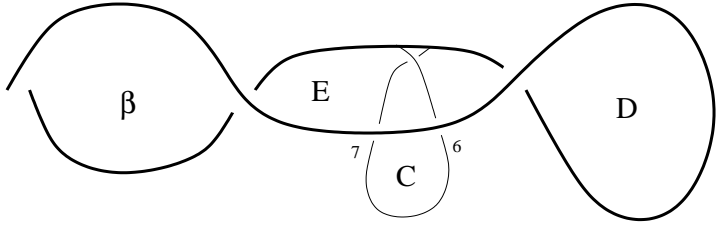} 
\caption{Sliding a large area $C(v,t)$.}
\label{fig:Cbig}
\end{figure}

\subsubsection{} The final stage of the fourth piece is shown as a Lagrangian slice in Figure \ref{fig:finalmid} and as the corresponding front in Figure \ref{fig:top1}. The latter serves as the starting point of the last piece.

\begin{figure}
\centering
\includegraphics[width=0.7\linewidth, angle=0]{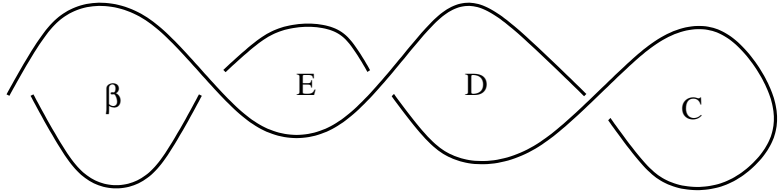} 
\caption{Final shape of the curves $\gamma_{v}$.}
\label{fig:finalmid}
\end{figure}

\begin{figure}
\centering
\includegraphics[scale=0.5]{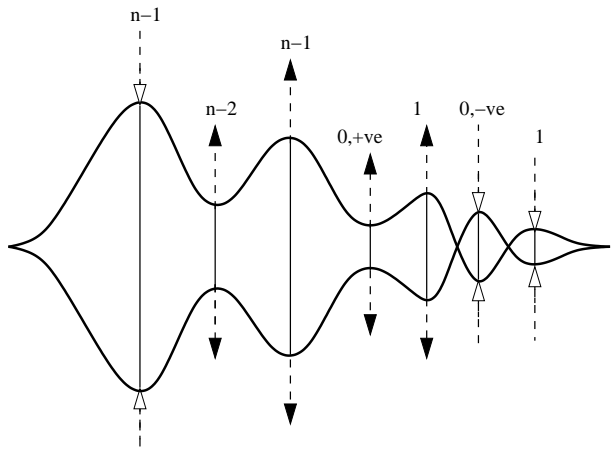} 
\caption{Initial position for the fifth piece.}
\label{fig:top1}
\end{figure}

\subsection{The fifth piece of the immersion}\label{ssec:5thpiece}
For the fifth, final piece of the immersion,  we return to the front representation.
The initial front of the fifth piece is a sphere with seven Reeb chords, see Figure \ref{fig:top1}, three that shrink and four that grow. It is the double of a function graph with critical points as follows: (i) the largest maximum is positive and shrinking; (ii) there is another positive local maximum that increases; (iii) one positive index $n-2$ point  increases; (iv) there are  two positive index $1$ chords,  one increasing and one decreasing;  (v) and there are two index $0$ chords, one positive increasing and one negative decreasing.

We pass the decreasing index $1$ chord through a double point,  making it increasing, see Figure \ref{fig:top2}. Then we cancel the two index $(0,1)$ pairs and the index $(n-1,n-2)$ pair of increasing chords. This leaves us with a standard sphere, with a decreasing slice Reeb chord corresponding to a maximum, that we cap off with a Lagrangian disk via a standard Morse modification; see Figure \ref{fig:top4}. 
This completes the construction of  the desired exact Lagrangian immersion $P\to\C^{n}$ with one double point. 
  
\begin{figure}
\centering
\includegraphics[width=0.7\linewidth, angle=0]{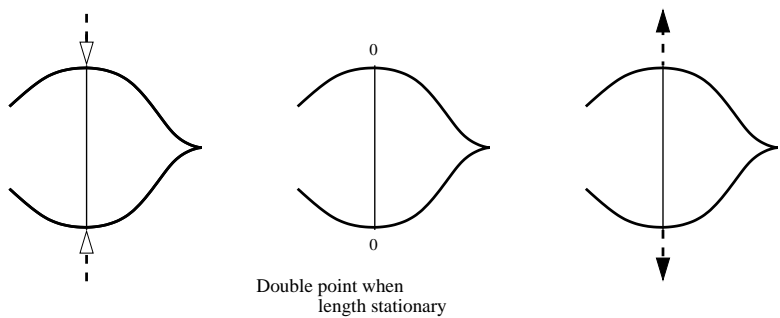} 
\caption{Slices passing through the double point.}
\label{fig:top2}
\end{figure}

\begin{figure}
\centering
\includegraphics[scale=0.4]{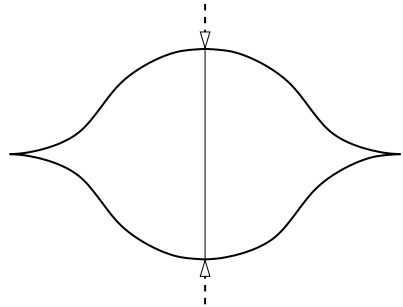} 
\caption{The final slice can be capped off with a maximum of $x_n$.
\label{fig:top4}}
\end{figure}

\end{document}